\documentclass{amsart}

\usepackage[utf8]{inputenc}
\usepackage{amsfonts}
\usepackage{enumerate}
\usepackage{amsmath}
\usepackage{amssymb}
\usepackage{dsfont}
\usepackage{xcolor}
\usepackage{amsthm}
\usepackage{cite}
\usepackage{pdflscape}
\usepackage{pgfplots}
\usepackage{graphicx}
\usepackage{subcaption}
\usepackage{mdframed}
\usepackage{pifont} 
\usepackage{mathrsfs}
\usepackage[export]{adjustbox}
\usepackage[colorlinks=true, allcolors=black]{hyperref}
\usepackage{cleveref}
\usepackage{csquotes}
\usepackage[english]{babel}
\usepackage{newtxtext}
\usepackage[a4paper,top=2cm,bottom=2cm,left=3cm,right=3cm,marginparwidth=1.75cm]{geometry}

\newtheorem{thm}{Theorem}[section]
\newtheorem{cor}[thm]{Corollary}
\newtheorem{ass}[thm]{Assumption}
\newtheorem{lem}[thm]{Lemma}

\theoremstyle{remark}
\newtheorem{rem}[thm]{Remark}

\newcommand{\e}{\textup{e}}
\newcommand{\EV}[1]{\mathbb{E} \left[ #1 \right]}

\newcommand{\LD}{\mathcal{L}^2(\mathcal{D})}

\newcommand{\dint}{\text{d}}
\newcommand{\inprod}[2]{\left\langle #1, #2 \right\rangle}
\newcommand{\norm}[1]{\left\| #1 \right\|}
\newcommand{\abs}[1]{\left| #1 \right|}
\pgfplotsset{compat=1.18}

\begin{document}

\title[Spectral approximation of stochastic evolution equations]{Spectral approximation of a class of stochastic time-fractional evolution equations}
\author{S. Knutsen Furset $^1$}
\thanks{$^1$ Department of Mathematical Sciences, NTNU, Trondheim, Norway}
\thanks{\ding{41} S. Knutsen Furset: simen.k.furset@ntnu.no}


\begin{abstract} A method for numerical approximation of a class of fractional parabolic stochastic evolution equations is introduced and analysed. This class of equations has recently been proposed as a space-time extension of the SPDE-method in spatial statistics. A truncation of the spectral basis function expansion is used to discretise in space, and then a quadrature is used to approximate the temporal evolution of each basis coefficient. Strong error bounds are proved both for the spectral and temporal approximations. The method is tested and the results are verified by several numerical experiments.
\end{abstract}

\maketitle 

\section{Introduction}
\label{sec:introduction} \subsection*{Motivation} We introduce a spectral method for approximating a class of stochastic fractional parabolic evolution equation
\begin{equation} \label{eq:fractional-time-heat}
\begin{cases}
    \left(\partial_t + A \right)^\gamma X(t) = \dot{W}_Q(t) \quad \text{for} \quad t > 0 \\
    X(t) = 0 \quad \text{for} \quad t \leq 0
\end{cases} \, ,
\end{equation}
where {$A$ is an unbounded linear operator such that} $-A$ generates a $C_0$-semigroup, {$Q$ is a bounded linear operator}, $W_Q$ is a (generalised) $Q$-Wiener process, $t \in [0,T]$, and $\gamma > \frac{1}{2}$. Our method consists of truncating an eigenfunction expansion of $\{X(t)\}_t$ and computing the coefficients by quadrature. The quadrature in time allows us to achieve an improvement in computational complexity compared to direct simulation using Cholesky decompositions. We are primarily interested in the case where $A$ and $Q^{-1}$ are elliptic operators defined by
\begin{equation} \label{eq:operators}
A = r^{-1}(\kappa^2 - \Delta)^{\alpha} \, , \quad Q = \sigma^2 r^{-2\gamma}(\kappa^2 - \Delta)^{-\beta} \, ,  
\end{equation}
where $\alpha > 0$, $\beta > 0$, $\sigma > 0$, $r > 0$, $\kappa > 0$ and $\Delta$ is the Laplacian {on $\mathcal{D}$}, equipped with zero Dirichlet boundary condition. (\ref{eq:fractional-time-heat}) with these operators has recently been suggested as a way to extend the so-called SPDE-method in spatial statistics to the case where the random fields are varying in time \cite{LindgrenBakka}. The SPDE-method, i.e., using stochastic partial differential equations (SPDE's) to construct spatial covariance models, is a popular approach in spatial statistics and has seen applications in a wide range of fields, see for example \cite{Bakka2018, Fuglstad2014, lindgren2011, Lindgren2022}. Motivating a statistical model through an SPDE creates a bridge between statistics and numerics by allowing the use of numerical methods developed for PDE's in the construction of covariance models. Compared to the classical method of specifying covariance functions directly, the SPDE-method allows easy generalisation to non-Euclidean spaces, for example the sphere. This is because covariance functions need to be non-negative definite, which is a difficult property to verify for functions on non-Euclidean spaces. However, covariance functions that derive from SPDE's automatically satisfy this requirement.

Solutions $\{X(t)\}_t$ to (\ref{eq:fractional-time-heat}) with the operators in (\ref{eq:operators}) have several properties that makes them useful for the modelling of spatio-temporal phenomena. First, it has been shown that $\{X(t)\}_t$ has spatial covariance operator that converges to a Whittle-Mátern operator as $t \rightarrow \infty$~\cite[Corollary 4.2]{kirchner2022}. This is of interest because Mátern fields are the most commonly used covariance models in spatial statistics; see for example \cite{Bakka2018, Cressie2011-it}. The alternative time-fractional SPDE $\partial_t^\gamma X(t) + A X(t) = \dot{W}_Q(t)$ has also been studied; see \cite{Bonaccorsi2009, Desch2011}, but solutions to this SPDE are not asymptotically Whittle-Mátern. Secondly, the parameters $\gamma$, $\alpha$, $\beta$, $r$, $\kappa$, and $\sigma$ are few in number, and they allow independent control over several important properties of the process $\{X(t)\}_t$, namely the spatial smoothness $\nu_s$, the temporal smoothness $\nu_t$, the spatial range $r_s$, and the temporal range $r_t$. The smoothness is here a measure of the regularity of the solutions, interpreted in the sense that if the solution has a smoothness of $\nu$, it has $n := \lceil \nu \rceil - 1$ mean-square derivatives, and the $n$-th derivative is {mean-square} $(\nu - n)$-Hölder continuous (up to an $\epsilon$ if $\nu$ is an integer) on $[0,T]$. The range is a measure of the rate at which correlation decays to zero. It is possible to parameterise the SPDE in terms of these properties, see \cite{LindgrenBakka} or Section \ref{sec:numerics} for details.

\subsection*{Related work} Kirchner \& Willems \cite{kirchner2022} derived rigorous well-posedness and regularity results for (\ref{eq:fractional-time-heat}) when considered on a domain $\mathcal{D} \subset \mathbb{R}^d$ with zero Dirichlet boundary condition. They showed that (\ref{eq:fractional-time-heat}) has a mild solution given by a stochastic convolution, see (\ref{eq:stochastic-convolution}) below. They also studied the covariance properties of (\ref{eq:fractional-time-heat}). A consequence of their main result~\cite[Theorem 3.12]{kirchner2022} is that if we use the operators $A$ and $Q$ in (\ref{eq:operators}), then $\{X(t)\}_t$ has $\mathcal{L}^2$-solutions when $\beta + (2\gamma - 1)\alpha - \frac{d}{2} > 0$ and the solution has spatial smoothness $\beta + (2\gamma - 1)\alpha - \frac{d}{2}$ and temporal smoothness $\gamma - \frac{1}{2} + \frac{1}{2\alpha}\min(\beta - \frac{d}{2}, 0)$. More recently, Kirchner \& Willems has also defined weak Markov properties for (\ref{eq:fractional-time-heat}) when $\gamma \in \mathbb{N}$ \cite{kirchner2023}.

Lindgren et. al. \cite{LindgrenBakka} used Fourier techniques to study the regularity properties of (\ref{eq:fractional-time-heat}) on $\mathbb{R}^d$. They introduced a weak finite element discretisation of (\ref{eq:fractional-time-heat}), however this method does not work for $\gamma \notin \mathbb{N}$. They suggested a reparametrisation for $\{X(t)\}_t$ in terms of $\nu_s$, $\nu_t$, $r_s$, $r_t$, and a ``non-separability'' parameter $\beta_s$. This demonstrates the above point that (\ref{eq:fractional-time-heat}) allows full control over the smoothness and range properties of $\{X(t)\}_t$ both in space and in time. The remaining parameter $\sigma$ is a scaling parameter for the marginal variance of the solution. We will use this parametrisation later in Section \ref{sec:numerics} when we conduct numerical experiments. 

The study of the numerical approximation of semi-linear SPDE's is already a developed field; see for example \cite{Jentzen2011, Lord2014}. An early contribution was the study of the $L^2$ and $\dot{H}^{-1}$ errors of finite element approximations of general linear stochastic evolution equations with additive noise by Yan \cite{Yan2005}. Kovács et. al. improved on these results by deriving both strong error estimates \cite{Kovcs2009}, and weak error estimates \cite{Kovcs2011} of a more general finite element approximation of semi-linear stochastic evolution equations with additive noise. Kovacs, Lang, \& Peterson studied approximations of the covariance operator of a linear SPDE using a finite element approximation in space and a rational approximation of the semi-group $\{S(t)\}_t$ in time \cite{Kovcs2022}. Spectral methods have also been considered in the literature. For example, Fahim, Hausenblas, \& Kovács studied the spatial discretisation errors for a finite element method and a spectral Galerkin method for general stochastic integral equations of convolution type with Gaussian noise \cite{Fahim2022}. Lang \& Schwab considered the strong error of a spectral decomposition of the stochastic heat equation on the sphere driven by additive, isotropic Wiener noise \cite{Lang2015}. Cohen \& Lang studied strong, weak, and almost sure convergence for the stochastic wave equation using a spectral method on the unit sphere \cite{Cohen2022}. Lang \& Motschan-Armen recently studied a stochastic heat equation on the sphere using a spectral method in space and Euler-Maruyama in time \cite{Lang2024}. There are many others who have also worked on similar problems, see for example \cite{Andersson2015, Angulo2007, Baas2022, Brhier2013, Brhier2022, Breit2024, Debussche1999, Debussche2011, Kelbert2005, Kirchner2017, Kovcs2020, Prohl2021}.

Non-stochastic versions of (\ref{eq:fractional-time-heat}) has also been considered in the literature. For example Nyström \& Sande studied the properties of the operator $\left(\partial_t - \Delta\right)^s$ for $0 < s < 1$ \cite{Nystrm2016} and Stinga \& Torrea derived regularity for solutions of the equation $\left(\partial_t - \Delta\right)^s u(t) = f(t)$ also for $0 < s < 1$ \cite{Stinga2017}. This problem has also been studied for by Biswas, Leon-Contreras, \& Stinga \cite{Biswas2021, Biswas2020}, and by Litsgård \& Nyström \cite{Litsgrd2022}. More recently Willems studied the problem under the more general assumption $0 < s < \infty$ \cite{Willems2024}.

\subsection*{Advantages and disadvantages of spectral approximations} We are using a spectral approximation in this work. Finite element methods are also common in SPDE literature. The advantage of using a spectral method on (\ref{eq:fractional-time-heat}) is that we can easily use any value of $\alpha$ and $\beta$ also non-integer, constrained only by the desire that $\{X(t)\}_t$ has $\mathcal{L}^2$-solutions. Since fractional operators are non-local they are more costly to discretise using finite-element methods, see for example \cite{Bolin2019, Bolin2018, Bolin2018_2, Bonito2019, Jansson2022}. In addition, spectral methods generally have very fast convergence when solutions are smooth; we see later that this also is the case for the SPDE in (\ref{eq:fractional-time-heat}). A drawback of the spectral approach is that we need the eigenfunctions of $A$ and $Q$ on the domain $\mathcal{D}$. In the case of the operators in (\ref{eq:operators}) the eigenfunctions are equal to the eigenfunctions of the Laplacian, which are known in some special cases, for example rectangles in $\mathbb{R}^d$ and the sphere $S^2$, both of which are important in applications. In the numerical section of the work we therefore restrict our attention to these two cases. However, in general one would need to numerically compute the eigenfunctions, which can be computationally expensive. This makes it difficult to extend the method to more complicated domains. It also makes it difficult to use more complicated operators, e.g., operators with spatially varying coefficients.

\subsection*{Contributions} To the authors knowledge, no one has yet considered methods for numerical approximation of (\ref{eq:fractional-time-heat}) for fractional $\gamma$. We therefore introduce a new numerical method for approximating (\ref{eq:fractional-time-heat}), based on discretising a Duhamel-type representation formula for the solutions derived in \cite{kirchner2022}. The representation formula is in the form of a stochastic convolution. Our method is spectral in space, we approximate $X(t)$ by truncating its eigenfunction expansion. The spectral approximation is standard, and the main contribution of this work is a quadrature method for approximating the resulting Itô integrals on each spectral basis function. The eigenfunctions and eigenvalues are assumed known. We prove strong error bounds for both the spectral and temporal discretisations and then verify these error estimates with numerical experiments. We also do simulations of $\{X(t)\}_t$ on the unit sphere, providing the first visualisation of the solutions of (\ref{eq:fractional-time-heat}) for fractional $\gamma$. This allows us to visually verify how the smoothness of the solutions depend on the parameters. Though we are primarily interested in the spatial operators in (\ref{eq:operators}), we achieve a more general theory with little extra effort by using slightly more general assumptions on $A$ and $Q$. These assumptions can be found in Assumption \ref{assumption:operator} below. One weakness of our approach is that the convergence rate of our temporal approximation goes to zero as $\gamma \rightarrow \frac{1}{2}$ due to a singularity in the Itô integrals. This is hard to avoid since quadrature methods with accumulating meshes are not readily applicable, see Remark \ref{remark:acc-mesh} for details.

\subsection*{Overview} In Section \ref{sec:spectral-approximation} we study the eigenfunction expansion of $X(t)$ and derive an Itô-integral describing the coefficient processes. In Section \ref{sec:discretisation} we describe our proposed discretisation schemes for space and time, and present results for the rate of convergence of these discretisations. Proofs of the results can be found in Section \ref{sec:proofs}. In Section \ref{sec:numerics} we verify the derived convergence rate with numerical examples.

\section{Spectral decomposition} \label{sec:spectral-approximation} 
\noindent {As mentioned above (\ref{eq:fractional-time-heat}) has mild solution given by the stochastic convolution}
\begin{equation} \label{eq:stochastic-convolution}
X(t) :=  \frac{1}{\Gamma(\gamma)} \int_0^t (t-s)^{\gamma - 1} S(t - s) \,\textup{d} W_Q(s) \, , \quad t \geq 0 \, ,
\end{equation} 
where $\{S(t)\}_t$ is the $C_0$-semigroup generated by {$-A$} \cite{kirchner2022}. This stochastic convolution is the basis for our numerical approximation. We make the following assumptions

\begin{ass} \label{assumption:operator} Let $A: \mathcal{L}^2(\mathcal{D}) \supset D(A) \rightarrow \mathcal{L}^2(\mathcal{D})$, {where $D(A)$ is the domain of $A$,} and $Q: \mathcal{L}^2(\mathcal{D}) \rightarrow \mathcal{L}^2(\mathcal{D})$. We assume the following:

\begin{enumerate}[(i)]
  \item $A$ and $Q$ are {linear} operators with a joint orthonormal basis of eigenfunctions $\{e_k\}_{k = 1}^{\infty}$, spanning $\mathcal{L}^2(\mathcal{D})$, with corresponding {real} eigenvalues $\{\mu_k\}_{k = 1}^{\infty}$ and $\{\lambda_k\}_{k = 1}^{\infty}$ respectively. \smallskip
  \item There exists an $r_{\mu} > 0$ and a $C_1 > 0$ such that $\{\mu_k\}_{k = 1}^\infty$ satisfies 
  \begin{equation*}
    \mu_k \geq C_1 k^{r_{\mu}}, 
  \end{equation*}
  where the eigenvalues $\{\mu_k\}_{k = 1}^{\infty}$ are indexed so that they are non-decreasing in $k$. \smallskip
  \item There exists an $r_{\lambda} < 0$ and a $C_2 > 0$ such that $\{\lambda_k\}_{k = 1}^\infty$ satisfies
  \begin{equation*}
      0 < \lambda_k \leq C_2 k^{r_{\lambda}} .
  \end{equation*}
  where the eigenvalues $\{\lambda_k\}_{k = 1}^{\infty}$ are indexed so that they are non-increasing in $k$.
\end{enumerate} 
\end{ass}
\noindent Under Assumption \ref{assumption:operator} it follows from~\cite[Theorem 3.12]{kirchner2022} that if $\gamma > \frac{1}{2}$, we have $\mathcal{L}^2(\mathcal{D})$-existence of (\ref{eq:fractional-time-heat}) if $1 + r_\lambda - r_\mu(2 \gamma - 1) < 0$. Additionally, $-A$ generates a $C_0$-semigroup $S(t) = \e^{-tA}$ on $\mathcal{L}^2(\mathcal{D})$ defined by $S(t) e_k := \e^{ - \mu_k t} e_k$~\cite[Section 6.1]{lectureNoteKovacsLarsson}. We denote
$$\inprod{u}{v} := \int_{\mathcal{D}} u v \, dx \, , \quad \norm{u}^2 := \inprod{u}{u} \, .$$
Since the eigenfunctions $\{e_k\}_{k = 1}^\infty$ span $\mathcal{L}^2(\mathcal{D})$, and since $X(t)$ is almost surely in $\mathcal{L}^2(\mathcal{D})$, we can for each $t \in [0,T]$ expand the solution $X(t)$ of (\ref{eq:stochastic-convolution}) in terms of the eigenfunctions of $A$, so that
\begin{equation} \label{eq:eigenfunction-expansion}
X(t) := \sum_{k = 1}^{\infty} c_k(t) e_k \, ,   
\end{equation}
almost surely in $\norm{\cdot}_{\mathcal{L}^2(\mathcal{D})}$, where, for each $k$, $c_k$ is a $[0,T]$-indexed stochastic processes on $\mathbb{R}$. In the following lemma we apply (\ref{eq:stochastic-convolution}) to calculate $c_k(t)$.

\begin{lem} \label{lemma:coefficient-representation} Assume that Assumption \ref{assumption:operator} holds and that $\{X(t)\}_t$ satisfies (\ref{eq:stochastic-convolution}). Define $\omega_k(t) := \frac{1}{\sqrt{\lambda_k}} \inprod{W_Q(t)}{e_k}$. Then {for all $t \in [0,T]$ and} for all $k = 1, 2, 3, ...$,
\begin{equation} \label{eq:coefficient-integral}
c_k(t) = \frac{\sqrt{\lambda_k}}{\Gamma(\gamma)} \int_0^t \e^{-\mu_k (t-s)} (t-s)^{\gamma - 1} \text{d} \omega_k(s) \, , 
\end{equation}
almost surely, where $c_k$ is as defined in (\ref{eq:eigenfunction-expansion}). {Furthermore, the random processes $\{c_k\}_{k = 1}^\infty$ are pairwise independent $\mathbb{R}$-valued, standard Wiener-processes.}
\end{lem}

\begin{proof} {By (\ref{eq:stochastic-convolution}) we have that 
\begin{align*}
    c_k(t) & = \sum_{j = 1}^{\infty} c_j(t) \inprod{e_j}{e_k} = \inprod{\sum_{j = 1}^{\infty} c_j(t) e_j}{e_k} = \inprod{X(t)}{e_k} \\
    & = \inprod{\frac{1}{\Gamma(\gamma)} \int_0^t (t-s)^{\gamma - 1} S(t - s) \textup{d} W_Q(s)}{e_k} \\
    & = \frac{1}{\Gamma(\gamma)} \int_0^t (t-s)^{\gamma - 1} \inprod{\textup{d} W_Q(s)}{S(t - s) e_k} \\
    & = \frac{1}{\Gamma(\gamma)} \int_0^t \e^{-\mu_k (t-s)} (t-s)^{\gamma - 1} \inprod{\textup{d} W_Q(s)}{e_k} \\
    & = \frac{\sqrt{\lambda_k}}{\Gamma(\gamma)} \int_0^t \e^{-\mu_k (t-s)} (t-s)^{\gamma - 1} \text{d} \omega_k(s) \, ,
\end{align*}
where the fifth step is justified by~\cite[Proposition 4.30]{daprato2014}. Furthermore, by definition $\{\omega_k(t)\}_{k = 1}^{\infty}$ are $\mathbb{R}$-valued Wiener-processes. The $\{\omega_k(t)\}_{k = 1}^{\infty}$'s are also independent, standard Wiener-processes since
\begin{align*}
    \EV{\omega_k(t) \omega_j(s)} & = \frac{1}{\sqrt{\lambda_k \lambda_j}} \EV{\inprod{W_Q(t)}{e_k} \inprod{W_Q(s)}{e_j}} \\ 
    & = \frac{1}{\sqrt{\lambda_k \lambda_j}} \inprod{\min(t,s)Q^\frac{1}{2}e_k}{Q^{\frac{1}{2}}e_j} \\ 
    & = \min(t,s) \delta_{kj} \, .\qedhere
\end{align*}}
\end{proof}

\section{Discretisation of the stochastic evolution equation \eqref{eq:fractional-time-heat}} \label{sec:discretisation}

In this section we discuss how to discretise problem \eqref{eq:fractional-time-heat}
in Section \ref{sec:heurestic}.  In Section \ref{sec:spectral-convergence} we derive a bound on the spectral approximation error, and in Section \ref{sec:method1} we derive an error bound for our temporal scheme. In Section \ref{sec:main-result} we combine these results to get an error bound for the full discretisation and in Section \ref{sec:complexity} we discuss the complexity of the method.

\subsection{The numerical method
} \label{sec:heurestic}
We truncate 
the spectral expansion (\ref{eq:eigenfunction-expansion}) of the solution $X(t)$ of the SPDE \eqref{eq:fractional-time-heat} given by the stochastic convolution formula \eqref{eq:stochastic-convolution},
\begin{equation} \label{eq:eigenfunction-truncation}
\Tilde{X}_M(t) := \sum_{k = 1}^{M} c_k(t) e_k \, .
\end{equation}
We then approximate the coefficient processes $c_k(t)$ by replacing $(t-s)^{\gamma - 1}$ in (\ref{eq:coefficient-integral}) by a piecewise polynomial $p_t(s)$ (of order $m$) on {a} discretisation grid {$0 = t_0 < t_1 < .... < t_N = T$, with $h := \max_{\ell}(t_\ell - t_{\ell - 1})$}, i.e.
\begin{equation*} \label{eq:coefficient-pol-approximation} \Tilde{c}_k^{(m)}(t) := \frac{\sqrt{\lambda_k}}{\Gamma(\gamma)} \int_0^t \e^{-\mu_k (t-s)} p_t(s) \text{d} \omega_k(s) \, .
\end{equation*}
The full discretisation of the solution $X(t)$ at $t_n$ in the discretisation grid is then:
\begin{align}
    \Tilde{X}^{h}_M(t_n) &:= \sum_{k = 1}^{M} \Tilde{c}^{(m)}_k(t_n) e_k \nonumber \\
    & = \sum_{k = 1}^{M} \frac{\sqrt{\lambda_k}}{\Gamma(\gamma)} \left( \int_0^{t_n} \e^{-\mu_k (t_n-s)} p_{t_n}(s) \text{d} \omega_k(s) \right) e_k \nonumber \\ 
    & = \sum_{k = 1}^{M} \sum_{\ell = 1}^n \frac{\sqrt{\lambda_k}}{\Gamma(\gamma)} \left( \int_{{t_{\ell - 1}}}^{t_\ell} \e^{-\mu_k (t_n-s)} p_{t_n}(s) \text{d} \omega_k(s) \right) e_k \nonumber \\ 
    & = \frac{1}{\Gamma(\gamma)} \sum_{k = 1}^M \sum_{\ell = 1}^n \sum_{j = 0}^m \sqrt{\lambda_k} \e^{-\mu_k (t_n-t_\ell)} b_{n, \ell, j} t_\ell^j {w_{\ell,j,k}} e_k  \, , \label{eq:full-disc-sum}
\end{align}
where $\sum_{j = 0}^m b_{n, \ell, j} s^j := {p_{t_n}(s)} \big|_{s \in [t_{\ell - 1}, t_\ell)}$, and
\begin{equation} \label{eq:stochastic-innovation-integral}
  {w_{\ell,j,k}} := \int_{t_{\ell - 1}}^{t_\ell} \left( \frac{s}{t_\ell} \right)^j \e^{-\mu_k (t_\ell-s)} \text{d} \omega_k(s) \, .
\end{equation}
{For fixed $\ell$ and $k$, $(w_{\ell,0,k}, w_{\ell,1,k}, ..., w_{\ell,m,k})$ is a multivariate Gaussian random variable with mean vector $0$ and covariance matrix $\Sigma$ defined by}
$$\Sigma_{i,j} := \int_{t_{\ell - 1}}^{t_\ell} \left( \frac{s}{t_\ell} \right)^{i + j} \e^{- 2 \mu_k (t_\ell - s)} \text{d}s \, .$$
The random vectors are independent in $\ell$ and $k$ and have bounded variance over $\ell$ and $k$ and and can therefore be simulated fast with complexity $O(m^3)$ using standard methods. An additional benefit is that the components of the covariance matrix $\Sigma_{i,j}$ can be computed exactly using integration by parts. This procedure therefore gives us a practical method to simulate $\Tilde{X}^{h}_M(t_n) \approx X(t_n)$.

\vspace{2 mm}

To use our method in practice, one must also choose a method for selecting the piecewise polynomial $p$. The simplest is to use a constant approximation and choose $p_t(s) \big|_{s \in [t_{\ell - 1}, t_\ell)} = (t - t_{\ell - 1} - \theta(t_{\ell} - t_{\ell - 1}))^{\gamma - 1}$ for some $\theta \in [0,1]$. Another option is to use an $\mathcal{L}^2$-projection to minimise $\int_{t_{\ell - 1}}^{t_{\ell}} (p_t(s) - (t - s)^{\gamma - 1})^2 \, \dint s$ on each interval $[t_{\ell - 1}, t_{\ell}]$. Lemma \ref{lem:projection-computation} below gives a practical to compute the coefficients of these projection polynomials.
 
\subsection{Convergence of the spectral truncation (\ref{eq:eigenfunction-truncation})} \label{sec:spectral-convergence}

\begin{thm} \label{thm:spectral-convergence} Assume Assumption \ref{assumption:operator} holds, and let $X$ and $\Tilde{X}$ be as defined in (\ref{eq:eigenfunction-expansion}) and (\ref{eq:eigenfunction-truncation}) respectively. If $1 + r_\lambda - (2 \gamma-1) r_\mu < 0$, then for all $t \in [0,T]$ we have $X(t) = \sum_{k = 1}^{\infty} c_k(t) e_k \in \mathcal{L}^2(\Omega, \mathcal{L}^2(\mathcal{D}))$ and
$$\EV{\norm{X(t) - \Tilde{X}_M(t)}^2} \leq C M^{-\frac{2 \nu}{d}} \, ,$$
where $C = \frac{C_2  \Gamma(2 \gamma - 1)}{C_1^{ 2\gamma-1}\Gamma(\gamma)^2 } \frac{d}{2 \nu}$, and $\nu = -\frac{d}{2}(1 + r_\lambda - (2\gamma-1) r_\mu)$.
\end{thm}
The proof of this theorem can be found in Section \ref{sec:proofs}.
We are especially interested in the case of the operators in (\ref{eq:operators}), i.e., $A = r^{-1}(\kappa^2 - \Delta)^{\alpha}$ and $Q = \sigma r^{-\gamma}(\kappa^2 - \Delta)^{-\beta}$. To analyse this case we use the following lemma, adapted from~\cite[Theorem 6.3.1]{davies1995}.

\begin{lem} \label{thm:laplacian_eigenproperties} Let $\mathcal{D} \subset \mathbb{R}^d$ be an open, non-empty, bounded, and connected domain. Let $\Delta := \sum_{k = 1}^d \frac{\partial^2}{\partial x_k^2}$ be defined on a subset $D(\Delta)$ of $\LD$ of sufficiently smooth functions which are zero on $\partial \mathcal{D}$, the boundary of $\mathcal{D}$. Then the operator $A := -\Delta$ has a set of orthonormal eigenfunctions $\{e_k\}_k$ spanning $\LD$, with a corresponding non-decreasing sequence of positive eigenvalues $\{\xi_k\}_k$ satisfying the estimates
$$C_\Delta k^{\frac{2}{d}} \leq \xi_k \leq D_\Delta k^{\frac{2}{d}} \, ,$$
for constants $C_\Delta > 0$ and $D_\Delta > 0$.
\end{lem}

This lemma can be used to show that the operators in (\ref{eq:operators}) satisfy Assumption \ref{assumption:operator} with $r_\mu = \frac{2 \alpha}{d}$, and $r_\lambda = - \frac{2 \beta}{d}$. The criterion $1 + r_\lambda - r_\mu(2 \gamma - 1) < 0$ for existence of solutions to (\ref{eq:stochastic-convolution}) then becomes $\beta + (2\gamma - 1)\alpha - \frac{d}{2} > 0$. In addition we can use~\cite[Theorem 3.12]{kirchner2022} to find that $\{X(t)\}_{t \in [0,T]}$ has at most $\beta + (2\gamma - 1)\alpha - \frac{d}{2}$ spatial {mean-square} derivatives and at most $\gamma - \frac{1}{2} + \frac{1}{2\alpha}\min(\beta - \frac{d}{2}, 0)$ temporal {mean-square} derivatives.

\begin{cor} \label{cor:laplacian-spectral-convergence} Let $\mathcal{D} \subset \mathbb{R}^d$ be a bounded domain. Assume $\alpha>0$, $\beta>0$, $\kappa > 0$, $r >0$, $\sigma > 0$, and $\gamma>\frac12$. Let $A = r^{-1}(\kappa^2 - \Delta)^{\alpha}$ and $Q = \sigma^2 r^{-2\gamma}(\kappa^2 - \Delta)^{-\beta}$ and assume that $\{X(t)\}_{t \in [0,T]}$ satisfies (\ref{eq:stochastic-convolution}). Assume further that $\beta + (2\gamma - 1) \alpha - \frac{d}{2} > 0$. Then for all $t \in [0,T]$
$$\EV{\norm{X(t) - \Tilde{X}_M(t)}^2} \leq C M^{-\frac{2 \nu}{d}} \, ,$$
where $C = \frac{\sigma r^{\gamma - 1} \Gamma(2 \gamma - 1)}{C_{\Delta}^{\nu + \frac{d}{2}} \Gamma(\gamma)^2} \frac{d}{2 \nu}$, and $\nu = \beta + (2\gamma - 1) \alpha - \frac{d}{2}$.
\end{cor}

\begin{proof} By Theorem \ref{thm:laplacian_eigenproperties} $-\Delta$ has an orthonormal basis of eigenfunctions $\{e_k\}_{k = 1}^{\infty}$ with corresponding eigenvalues $\{\xi_k\}_{k = 1}^{\infty}$, satisfying $C_\Delta k^{\frac{2}{d}} \leq \xi_k$ by Theorem \ref{thm:laplacian_eigenproperties}. Thus $A$ and $Q$ have eigenvalues $\mu_k = r^{-1}(\kappa^2 + \xi_k)^{\alpha}$ and $\lambda_k = \sigma r^{-\gamma}(\kappa^2 + \xi_k)^{-\beta}$ and $\{\mu_k\}_{k = 1}^\infty$ and $\{\lambda_k\}_{k = 1}^{\infty}$ satisfy the bounds $r^{-1} {C_\Delta^{\alpha}} k^{\frac{2 \alpha}{d}}\leq \mu_k$ and $\lambda_k \leq \sigma r^{-\gamma} {C_\Delta^{-\beta}} k^{- \frac{2 \beta}{d}}$. $A$ and $Q$ therefore fulfill Assumption \ref{assumption:operator} with $C_1 = r^{-1} {C_\Delta^{\alpha}}$, $C_2 = \sigma r^{-\gamma} {C_\Delta^{-\beta}}$, $r_\mu = \frac{2 \alpha}{d}$, and $r_\lambda = - \frac{2 \beta}{d}$. We can then use Theorem \ref{thm:spectral-convergence} to conclude.
\end{proof}

\begin{rem}
We see that in Corollary \ref{cor:laplacian-spectral-convergence} the parameter $\nu$ has the interpretation of being exactly the spatial smoothness of the solution $\{X(t)\}_{t \in [0,T]}$. This is the motivation for introducing the parameter also in Theorem \ref{thm:spectral-convergence}, where it may seem unnatural.
\end{rem}

\begin{rem} It is straightforward to extend this corollary to the case $A = r^{-1}(\kappa_1^2 - \Delta)^{\alpha}$ and $Q = \sigma^2 r^{-2\gamma}(\kappa_2^2 - \Delta)^{-\beta}$, where $\kappa_1 \neq \kappa_2$. We have avoided this in order to reduce notational clutter.
\end{rem}

\begin{rem}
If we substitute the Laplace-Beltrami operator for the Laplace operator, then the above corollary also holds for the sphere. In fact it holds for any compact Riemannian manifold. This is because Assumption \ref{assumption:operator} holds also in this setting by~\cite[Theorem 4.3.1]{Lablee2015-pi} and~\cite[Theorem 7.6.4]{Lablee2015-pi}.
\end{rem}

\subsection{Approximation of the coefficients processes $c_k(t)$ for $t\in(0,T]$}
\label{sec:method1}
To simplify notation we define
\begin{equation*} \label{eq:fractional-function}
  f_t(s) := (t - s)^{\gamma - 1} \, .   
\end{equation*}
We make the following definitions and assumptions:
\begin{ass} \label{assumption:pm} {Let $t \in [0,T]$}. Let $\mathcal{I} := \{t_\ell\}_{\ell = 0}^N$ be a discretisation of $[0,T]$ such that $0 = t_0 < t_1 < ... < t_N = T$. Define 
$$\mathcal{P}_m(\mathcal{I}) := \{p \in \mathcal{L}^2([0,{t)}) \, : \, p|_{[t_{\ell - 1}, t_\ell)} \in \mathbb{P}_m([t_{\ell - 1}, t_\ell)) \quad \forall \quad \ell = 1,2,...,N \} \, ,$$
where $\mathbb{P}_m([a,b))$ is the space of polynomials of order {at most} $m$ on the interval $[a,b)$. Let $\Pi_m: \mathcal{L}^2([0,{t)}) \supset \mathcal{A} \rightarrow \mathcal{P}_m(\mathcal{I})$, and assume
\begin{enumerate}[(i)]
    \item $\Pi_m g = g$ for all $g \in \mathcal{P}_m(\mathcal{I}) \subset \mathcal{L}^2([0,{t)})$.\smallskip
    \item $\mathcal{B} := \{g(s) + \alpha f_t(s) | \alpha \in \mathbb{R}$, $g \in \mathcal{P}_m(\mathcal{I}) \} \subseteq \mathcal{A}$\smallskip
    \item There exists a $C_\Pi > 0$ such that $\norm{\Pi_m \chi}_{\mathcal{L}^2([0,t))} \leq C_\Pi \norm{\chi}_{\mathcal{L}^2([0,t))}$ for all $\chi \in \mathcal{B}$.
\end{enumerate}
\end{ass}

We make (\ref{eq:full-disc-sum}) precise by letting $\Pi_m$ be as defined in Assumption \ref{assumption:pm} and defining the approximation $\Tilde{c}_k^{(m)}$ of the spectral coefficient $c_k$ in \eqref{eq:coefficient-integral} by
\begin{equation} \label{eq:order-m-method} \Tilde{c}_k^{(m)}(t) := \frac{\sqrt{\lambda_k}}{\Gamma(\gamma)} \int_0^t \e^{-\mu_k (t-s)} \left( \Pi_m f_t \right) \text{d} \omega_k(s) \, , 
\end{equation}
The following theorem establishes error bounds for the approximation in (\ref{eq:order-m-method}). The proof can be found in Section \ref{sec:proofs}.

\begin{thm} \label{thm-convergence} {Let $\gamma > \frac{1}{2}$, and} assume Assumption \ref{assumption:pm} holds. Let $c_k(t)$ and $\Tilde{c}_k^{(m)}(t)$ be as defined in (\ref{eq:coefficient-integral}) and (\ref{eq:order-m-method}). Let $\mu_k > 0$, $\lambda_k > 0$, and $h := \max_{\ell}(t_\ell - t_{\ell - 1})$.
\begin{enumerate}[(a)] 
    \item Then $$\EV{\abs{c_k(t) - \Tilde{c}^{(m)}_k(t)}^2} \leq  \sqrt{\lambda_k} \delta_{\gamma, h} C(C_\Pi, t, \gamma)  \, h^{\min(2\gamma - 1, 2m + 2)} \, ,$$
    where $C(C_\Pi, t, \gamma)$ depends on $C_\Pi$, $t$, and $\gamma$ and 
    $$\delta_{\gamma,h} := \begin{cases}
    1 & \text{when} \, \gamma \neq m + \frac{3}{2} \, , \\
    1 +\sqrt{\log(t h^{-1})} & \text{when} \,  \gamma = m + \frac{3}{2} \, .
    \end{cases}$$
    \item Assuming additionally that $\gamma > m + \frac{3}{2} + \frac{1}{q}$ for some $q \in [2, \infty)$ and $\norm{\Pi_m \chi}_{\mathcal{L}^p([0,t))} \leq C_{\Pi,p} \norm{\chi}_{\mathcal{L}^p([0,t)}$ for all $\chi \in \mathcal{B}$, where $p \in (2,\infty]$ satisfies $\frac{1}{q} + \frac{1}{p} = \frac{1}{2}$, then
    $$\EV{\abs{c_k(t) - \Tilde{c}^{(m)}_k(t)}^2} \leq \sqrt{\lambda_k}\mu_k^{-\frac{1}{q}} C(C_{\Pi, p}, t, \gamma, q) h^{2m + 2} \, ,$$
    where the $C(C_{\Pi, p}, t, \gamma,q)$ depends on $C_{\Pi, p}$, $t$, $\gamma$, and $q$. 
\end{enumerate}
\end{thm}


\begin{rem} \label{remark:acc-mesh} When $\gamma < 1$ we have a singularity at $t$ in the function $f_t$. You might then expect to get a faster convergence when using a mesh with nodes that accumulate at the singularity. In our case, since $t$ is not fixed, we have a moving singularity, which makes it intractable to use an accumulating mesh. For Lebesgue integrals a change of variables would fix this problem, but this is not immediately applicable for the Itô integral in (\ref{eq:coefficient-integral}) since $\dint \omega_k(t)$ is almost surely not translation-invariant. 
\vspace{2mm}
    
\end{rem}

\subsection*{Case 1: Orthogonal $\mathcal{L}^2$-projection} One possible choice of $\Pi_m$ is the orthogonal $\mathcal{L}^2$-projection $\pi_m$ defined such that $\inprod{\pi_m f}{g} = \inprod{f}{g}$ for all $g \in \mathcal{P}_m(\mathcal{I})$.

\begin{cor} \label{cor:projection-method} Let $t \in [0,T]$, $\mu_k > 0$, $\lambda_k > 0$, $\gamma > \frac{1}{2}$, and $h := \max_{\ell}(t_\ell - t_{\ell - 1})$, and $\pi_m: \mathcal{L}^2([0,t)) \rightarrow {\mathcal{P}_m(\mathcal{I})}$ be the orthogonal $\mathcal{L}^2$-projection. Then
$$ \EV{\abs{c_k(t) - \Tilde{c}^{(m)}_k(t)}^2} \leq \sqrt{\lambda_k} C_{\gamma, t} \delta_{\gamma, h} \, h^{\min(2\gamma - 1, 2m + 2)} \, ,$$
where the constant $C_{\gamma, t}$ depends on $\gamma$ and $t$, and $\delta_{\gamma, h}$ is as defined in Theorem \ref{thm-convergence}. 
\end{cor}

\begin{proof} This follows immediately by Theorem \ref{thm-convergence} and the $\mathcal{L}^2$-boundedness of $\pi_m$.
\end{proof}
In our case doing computations with the $\mathcal{L}^2$-projection is tractable, as the following lemma shows. The proof of the lemma can be found in Section \ref{sec:proofs}.
\begin{lem} \label{lem:projection-computation} Define $s \mapsto (t - s)^{\gamma - 1}$ as a function on $[a,b) \subset [0, \infty)$. Let $\{q_l\}_{l = 0}^m$ be an orthonormal basis for $\mathbb{P}_m([a,b))$, with $q_l(s) := \sum_{j = 0}^m q_l^{(j)} s^j$. Let
$$i_k := \sum_{j = 0}^k \binom{k}{j} \frac{(-1)^j t^{k - j}}{\gamma + j} \left[ (t - a)^{\gamma + j} - (t - b)^{\gamma + j} \right] \, .$$
Denote by $p(s) := \sum_{j = 0} \alpha_j s^j$ the $\mathcal{L}^2$-projection of $s \mapsto (t - s)^{\gamma - 1}$ into $\mathbb{P}_m([a,b))$. Then
$$\alpha_j = \sum_{k = 0}^m \sum_{l = 0}^m q_l^{(k)} q_l^{(j)} i_k \, .$$
\end{lem}
\subsection*{Case 2: Piecewise constant interpolation} We are also interested in $\Pi_0: C([0,t_n)) {\cap \mathcal{L}^2([0,t_n))} \rightarrow \mathcal{P}_0(\mathcal{I})$ defined by
\begin{equation} \label{def:pi_0}
\Pi_0: f \mapsto \sum_{\ell = 1}^{n} f(t_{\ell - 1} + \theta (t_\ell - t_{\ell - 1})) \mathbb{I}_{[t_{\ell - 1}, t_\ell)}(s) \, ,
\end{equation}
where $\theta \in [0,1]$ and $t_n \in \mathcal{I}$. We can then rewrite $\Tilde{c}^{(0)}_k(t)$ as 
\begin{align} \label{eq:order0disc}
\Tilde{c}^{(0)}_k(t_n) &= \frac{\sqrt{\lambda_k}}{\Gamma(\gamma)} \sum_{\ell = 1}^n (t_n - t_{\ell - 1} - \theta h_\ell)^{\gamma - 1} \e^{-\mu_k (t-t_\ell)} \int_{t_{\ell - 1}}^{t_\ell} \e^{-\mu_k (t_\ell - s)} \, \text{d} \omega_k(s) \, ,
\end{align}
where $h_\ell := t_\ell - t_{\ell - 1}$. 
%
%
To show convergence we need the lemma below, the proof of which can be found in Section \ref{sec:proofs}.

\begin{lem} \label{lem:boundedness-p0} Let $t_n \in \mathcal{I}$ and let $\Pi_0$ be as defined in (\ref{def:pi_0}). Then, if either $\theta \in [0,1)$, or $\gamma > 1$ and $\theta \in [0,1]$, there exists $C_\Pi > 0$ such that $\norm{\Pi_0 \chi}_{\mathcal{L}^2([0,t_n))} \leq C_\Pi \norm{\chi}_{\mathcal{L}^2([0,t_n))}$ for all $\chi(s) = g(s) + \alpha f_{t_n}(s)$ where $\alpha \in \mathbb{R}$ and $g \in \mathcal{P}_0(\mathcal{I})$, and $C_\Pi$ depends only on $\theta$, $\gamma$, and ${\eta} := \frac{\max_\ell(h_\ell)}{\min_\ell(h_\ell)}$.
\end{lem}

\begin{cor} \label{cor:order-0-method} Let $t_n \in \mathcal{I}$, $\mu_k > 0$, $\lambda_k > 0$, $\gamma > \frac{1}{2}$, and $h := \max_{\ell}(h_\ell)$, where $h_{\ell} = t_\ell - t_{\ell - 1}$. Then for $n = 1,2,...N$, $\theta \in [0,1)$ we have that
$$\EV{\abs{c_k(t_n) - \Tilde{c}^{(0)}_k(t_n)}^2} \leq \sqrt{\lambda_k} C(\gamma, \theta, t, {\eta}) \delta_{\gamma, h} \, h^{\min(2\gamma - 1, 2)} \, ,$$
where $\Tilde{c}^{(0)}_k(t_n)$ is as defined in (\ref{eq:order0disc}), and $\delta_{\gamma, h}$ is as defined in Theorem \ref{thm-convergence}. The constant $C(\gamma, \theta, t, {\eta})$ depends on $\gamma$, $\theta$, $t$, and ${\eta} := \frac{\max_\ell(h_\ell)}{\min_\ell(h_\ell)}$. Additionally, if $\gamma > 1$, the results holds also for $\theta = 1$.
\end{cor}

\begin{proof} The result follows directly by application of Lemma \ref{lem:boundedness-p0} and Theorem \ref{thm-convergence}.
\end{proof}


\subsection{Convergence of the full discretisation}
\label{sec:main-result}

\vspace{2 mm}

We can derive an error bound for the full discretisation by combining Theorem \ref{thm:spectral-convergence} and Theorem \ref{thm-convergence}. A proof can found in Section \ref{sec:proofs}.

\begin{thm} \label{thm:main-result} Assume that Assumption \ref{assumption:operator} and Assumption \ref{assumption:pm} holds and that $\{X(t)\}_{t \in [0,T]}$ satisfies (\ref{eq:stochastic-convolution}). Assume $1 + r_\lambda - (2 \gamma-1) r_\mu < 0$ and let $t \in [0,T]$, and $\gamma > \frac{1}{2}$. Then 
$$\EV{\norm{X(t) - \Tilde{X}^{h}_M(t)}^2} \leq C^{'} M^{- \frac{2 \nu}{d}} + C^{''}_M \delta_{\gamma,h} h^{\min(2 \gamma - 1, 2m + 2)} \, ,$$
where $C^{'} = \frac{C_2  \Gamma(2 \gamma - 1)}{C_1^{ 2\gamma - 1} \Gamma(\gamma)^2 (1 + r_\lambda - (2 \gamma-1) r_\mu)}$, and $C^{''}_M = C(C_\Pi, t, \gamma) \sum_{k = 1}^M \lambda_k$, and $\nu = -\frac{d}{2}(1 + r_\lambda - (2\gamma - 1) r_\mu)$, and $\delta_{\gamma, h}$ is as defined in Theorem \ref{thm-convergence}. If in addition the assumptions of Theorem \ref{thm-convergence} b) holds, then $C^{''}_M = \frac{C}{\Gamma(\gamma)^2} \sum_{k = 1}^M \mu_k^{-\frac{2}{q}} \lambda_k$.

\end{thm}

Theorem \ref{thm:main-result} gives us the convergence of the full discretisation under reasonably general assumptions. The fact that $C_M^{''}$ is potentially unbounded in $M$ is a difficulty. In the case where $\text{Tr}(Q) = \sum_{k = 1}^{\infty} \lambda_k < \infty$ we can bound $C_M^{''}$ by $\text{Tr}(Q)$. If we have enough temporal smoothness, i.e., if $\gamma$ is large enough, we can use the second part of Theorem \ref{thm:main-result} to get the improved constant $C_M^{''} = \frac{C}{\Gamma(\gamma)^2} \sqrt{\sum_{k = 1}^M \mu_k^{-\frac{2}{q}} \lambda_k}$, allowing for a more degenerate $Q$. In many practical applications $A^{-1}$ will be trace-class, i.e., $\sum_{k = 1}^M \mu_k^{-1} < \infty$, so that in the extreme case of $\gamma > m + \frac{3}{2} + \frac{1}{2}$ we have convergence of $\sum_{k = 1}^M \mu_k^{-1} \lambda_k$ assuming only that $Q$ is bounded. Of course, we have two degrees of freedom, so even if $C_M^{''}$ is not bounded in $M$ the full discretisation still convergences, since we can first control the spatial error $C^{'} M^{-\frac{2 \nu}{d}}$ by making $M$ large and then control the temporal error $C^{''}_M \delta_{\gamma,h} h^{\min(2m + 2, 2\gamma - 1)} $ by making $h$ small. The following corollary tells us how small $h$ must be in order to achieve the optimal spatial convergence order $M^{-\frac{2\nu}{d}}$. A proof can found in Section \ref{sec:proofs}.

\begin{cor} \label{cor:opt-rate-balancing} Let the assumptions of Theorem \ref{thm:main-result} hold. Then choosing $h = M^{-\zeta}$, where
{\begin{equation*}
    \zeta = \begin{cases}
    r_\mu - \min(0, \frac{\delta}{b}) & \text{when} \, \gamma \leq m + 2 \, , \\[0.2cm]
    \frac{2 \gamma - 2}{2 m + 2} r_\mu - \min(0,\frac{\delta}{b}) & \text{when} \,  \gamma > m + 2 \, ,
\end{cases} 
\end{equation*}}
where $b = \min(2 \gamma - 1, 2 m + 2)$ and
$$\delta = \begin{cases}
    1 + r_\lambda & \text{when} \, \gamma \in (\frac{1}{2}, m + \frac{3}{2}] \, , \\
    1 + r_\lambda - (2 \gamma - 2m - 3)r_\mu & \text{when} \, \gamma \in (m + \frac{3}{2}, m + 2] \, , \\
    1 + r_\lambda - r_\mu & \text{when} \, \gamma \in (m + 2, \infty) \, ,
\end{cases} $$
then 
$$\EV{\norm{X(t) - \Tilde{X}^{h}_M(t)}^2} \lesssim M^{- \frac{2 \nu}{d}} \, ,$$
where $\nu$ is as defined in Theorem \ref{thm:spectral-convergence} and depends on $\gamma$, $r_\lambda$, and $r_\mu$. {In the special case of the operators in we have $r_\mu = \frac{2 \alpha}{d}$ and $r_\lambda = -\frac{2\beta}{d}$.}
\end{cor}

\subsection{Computational complexity and cost} \label{sec:complexity} {We are also interested in the computational complexity of the method. Consider a set of times $t_1^*, t_2^*, ..., t_J^*$, on which we are interested in simulating $X(t)$. We then apply our method on a potentially finer mesh $t_1, t_2, ... t_N$. Simulating the stochastic integrals in (\ref{eq:stochastic-innovation-integral}) and computing the sum in (\ref{eq:full-disc-sum}) then takes $Mnm^3 \leq MNm^3$ operations, so computing $X_M^h(t_\ell)$ for each $\ell = 0,...,J$ has complexity $O(JNM m^3)$, or $O(JNM)$ if we consider the method order $m$ as fixed. We can use this bound on the complexity and Corollary \ref{cor:opt-rate-balancing} to derive another bound on the computational complexity in terms of the number of operations needed required to achieve an error of $\epsilon$. A proof can found in Section \ref{sec:proofs}.
\begin{cor} \label{cor:comp-cost} Let the assumptions of Theorem \ref{thm:main-result} hold, and let $J$ be as defined above. The computational cost $\mathcal{C}$ required to achieve an error $\epsilon := \EV{\norm{X(t) - \Tilde{X}^{h}_M(t)}^2}$ can be bounded by
\begin{equation} \label{eq:comp-cost}
    \mathcal{C} \lesssim J \epsilon^{-\frac{d}{2\nu}(1 + \zeta)} \, ,
\end{equation}
where $\zeta$ is as defined in Corollary \ref{cor:opt-rate-balancing}, and $\nu$ is as defined in Theorem \ref{thm:spectral-convergence} (alternatively Corollary \ref{cor:laplacian-spectral-convergence}).
\end{cor}
For comparison, the 'naive' method of using a Cholesky decomposition of the covariance matrix to simulate (\ref{eq:coefficient-integral}) has a complexity of $O(J^3 M)$. The cost to achieve an error of $\epsilon$ is then
$$\mathcal{C}_{\text{chol}} \lesssim J^3 M \lesssim J^3 \epsilon^{-\frac{d}{2\nu}} \, .$$
This highlights both a strength and a weakness of the proposed method. If we consider $J$ fixed, then as the acceptable error $\epsilon \rightarrow 0$, Cholesky decomposition will eventually be cheaper than our method, due to the additional factor of $1 + \zeta > 1$ in the exponent in (\ref{eq:comp-cost}). However, our method scales linearly in $J$, so if we fix an error $\epsilon$ that we deem 'acceptable', then as $J \rightarrow \infty$ our method will eventually be cheaper than Cholesky decomposition. This improvement is sufficient for our method to allow for simulation for much larger $J$'s compared to what  is feasible with the standard method. Note for example that in Section \ref{sec:numerical-convergence-of-temporal-disc} below we simulate our reference solution on a reference mesh of $2^{17} = 131072$ nodes. This would have been completely unfeasible using Cholesky decomposition.

\begin{rem} Its likely possible to get complexity $O(JM)$ by replacing $\Sigma_{\ell = 1}^n$ with $\Sigma_{\ell = \max(1,n - d)}^n$ in (\ref{eq:full-disc-sum}) for some fixed lag $d$. Bounding the error of this additional approximation is beyond the scope of this work.
\end{rem}}

\section{Numerical examples} \label{sec:numerics}
We now conduct numerical experiments to demonstrate the convergence rate of our method. We first test the temporal convergence order established in Theorem \ref{thm-convergence}. We consider two specific methods: the one described in Corollary \ref{cor:order-0-method} with $\theta = 0$, and the one described in Corollary \ref{cor:projection-method} with $m = 1$. For the purpose of the numerical experiment we refer to these two methods respectively as the left-point method and the projection method. We also use a numerical experiment to demonstrate the spectral convergence order established in Theorem \ref{thm-convergence}. For this purpose we use the domain $[0,1]^2$ and spatial operators $A = r^{-1}(\kappa^2 - \Delta)^{\alpha}$ and $Q = \sigma^2 r^{-2\gamma}(\kappa^2 - \Delta)^{-\beta}$ based on the Laplacian. We also do some simulations on the unit sphere and use these simulations to demonstrate our ability to control the spatial and temporal smoothness separately.

Throughout the numerical experiment we will make use of the reparametrisation suggested in \cite{LindgrenBakka}. This reparametrisation specifies $A$ and $Q$ through the spatial smoothness $\nu_s > 0$, the temporal smoothness $\nu_t > 0$, the spatial range $r_s > 0$, the temporal range $r_t > 0$, and a ``non-separability'' parameter $0 \leq \beta_s \leq 1$. These parameters are given by
$$
\label{eq:reparametrisation}
\nu_s = \beta + (2\gamma - 1)\alpha - \frac{d}{2}, \quad
\nu_t = \gamma - \frac{1}{2} + {\frac{1}{2\alpha}}\min \left(\beta - \frac{d}{2}, 0 \right), \quad
\beta_s = \frac{(2\gamma - 1)\alpha}{\beta + (2\gamma - 1)\alpha},
$$
$$
r_s = \kappa^{-1} \sqrt{8 \nu_s}, \quad
r_t = r \kappa^{-2 \alpha} \sqrt{8 \left(\gamma - \frac{1}{2}\right)} \, ,
$$
or alternatively
\begin{equation} 
\gamma = \nu_t \max \left(1, \frac{\beta_s}{\beta^*} \right) + \frac{1}{2}, \quad \alpha = \frac{\nu_s}{2\nu_t} \min \left(1, \frac{\beta_s}{\beta^*} \right), \quad
\beta = \frac{1 - \beta_s}{\beta^*} \nu_s \, ,
\end{equation}
$$\kappa = \frac{\sqrt{8 \nu_s}}{r_s}, \quad r = \frac{r_t \kappa^{2 \alpha}}{\sqrt{8(\gamma - \frac{1}{2})}} \, ,$$
where $\beta^* = \frac{\nu_s}{\nu_s + \frac{d}{2}}$. The parameter $\sigma$ still controls the variance scaling.

\subsection{Testing the convergence of the temporal discretisation} \label{sec:numerical-convergence-of-temporal-disc}

\begin{figure}[p]
     \begin{subfigure}[b]{0.42\textwidth}
         \centering
         \includegraphics[width=\textwidth]{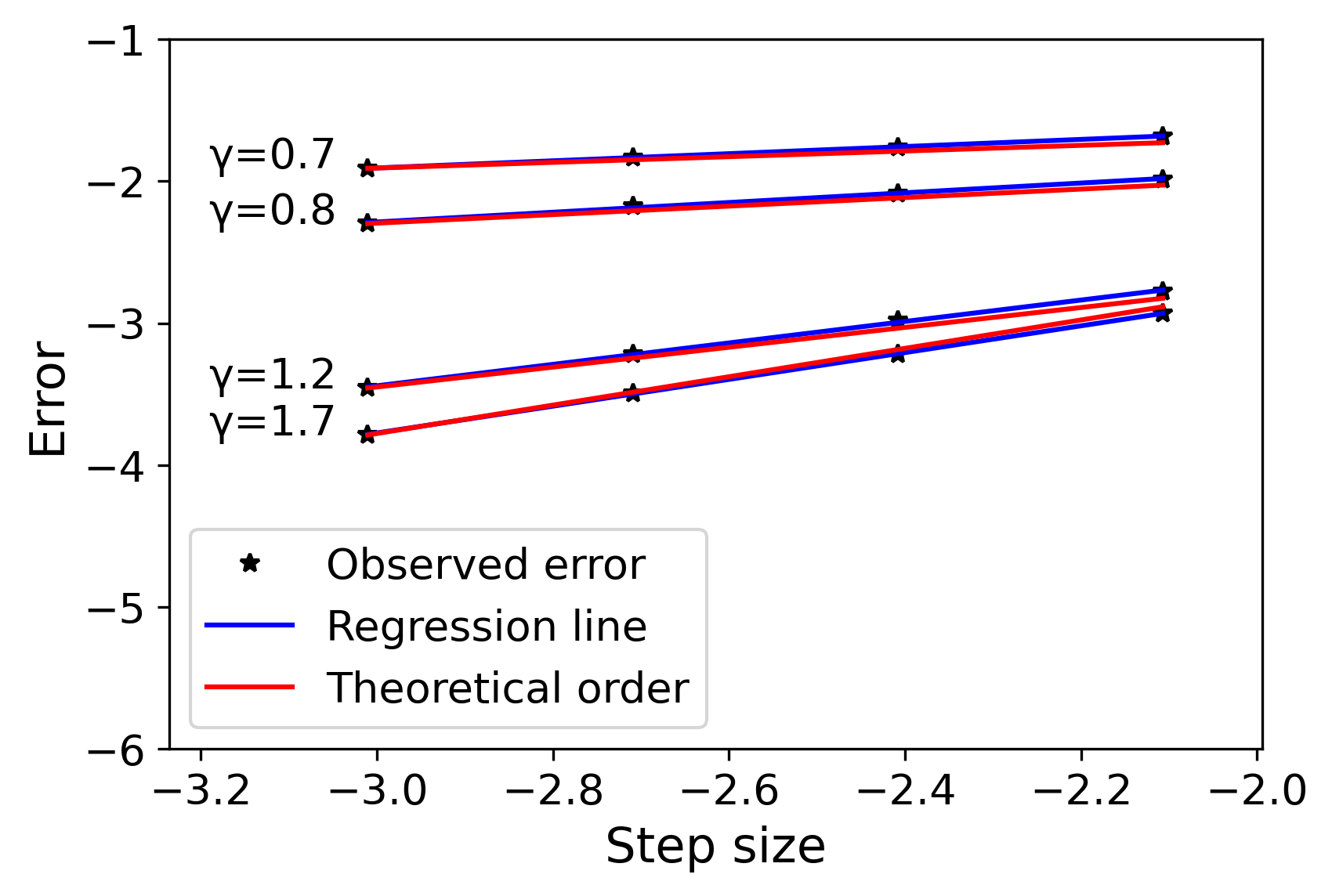}
     \end{subfigure}
     \begin{subfigure}[b]{0.42\textwidth}
         \centering
         \includegraphics[width=\textwidth]{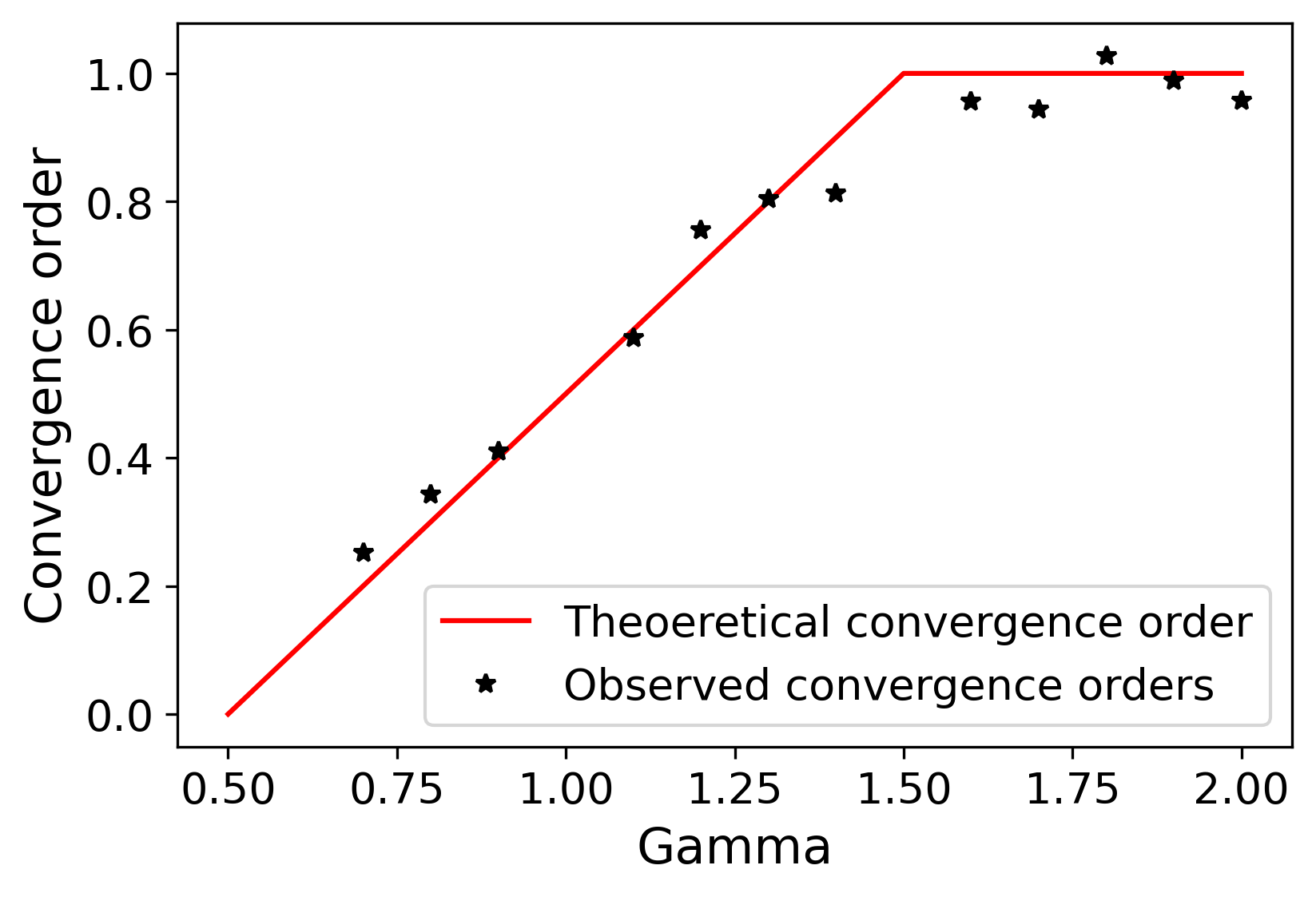}
         \label{fig:constant}
     \end{subfigure}
     \\
     \centering
     \caption{Left: Convergence plots for the left-point method for a selection of the values of $\gamma$. The regression line is a linear regression of all the observed errors. The theoretical order line emanates from the left-most point and has slope equal to the theoretical convergence order. Both axes are in $\log_{10}$-scale. Right: Plot of estimated convergence orders for the left-points method. The slope of the linear regressions of observed errors for all the tested values of $\gamma$ can be seen in black. The theoretical convergence order is also plotted.}
     \label{fig:constant_convergence_plots}
\end{figure}

\begin{figure}[p]
     \begin{subfigure}[b]{0.42\textwidth}
         \centering
         \includegraphics[width=\textwidth]{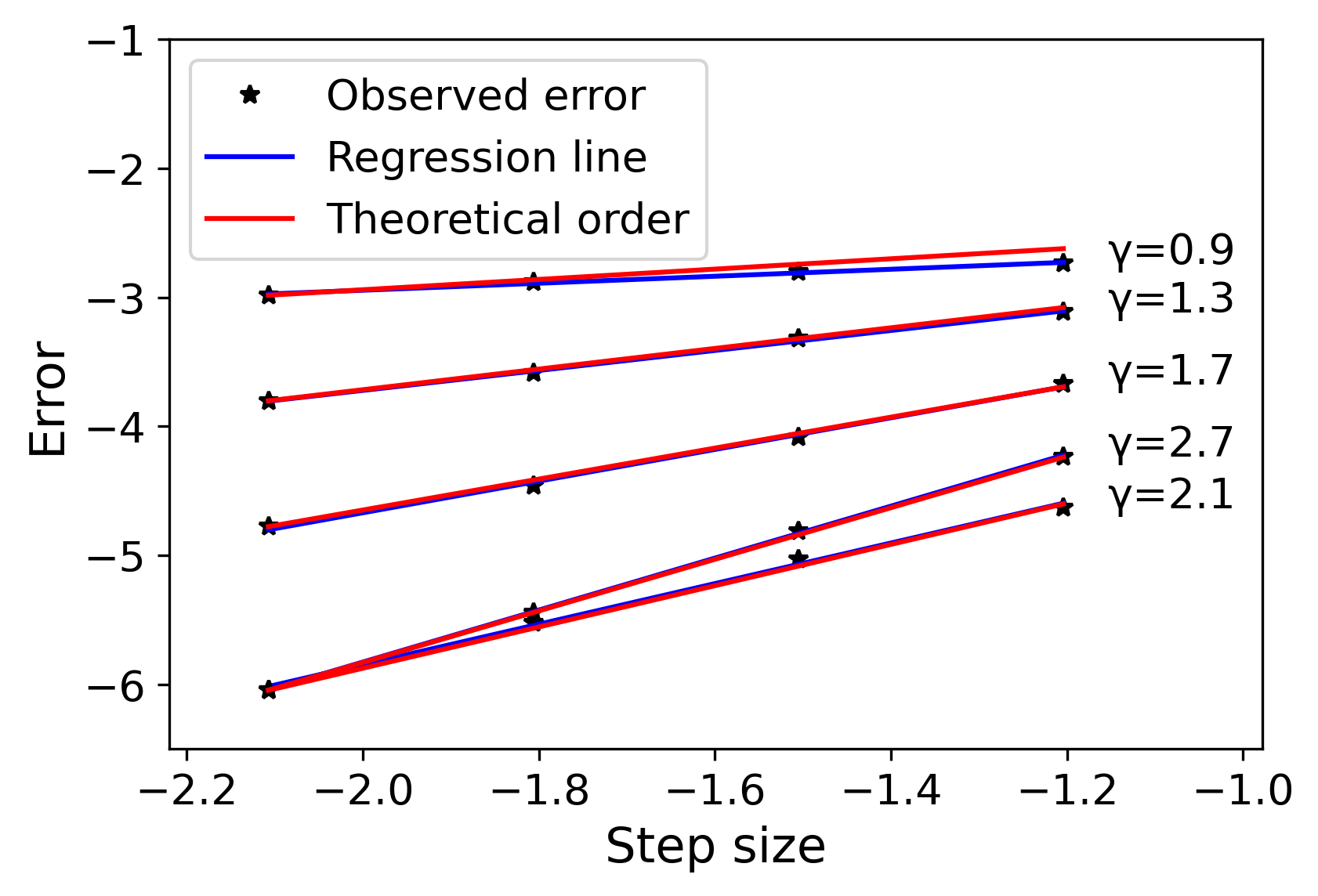}
         \label{fig:constant}
     \end{subfigure}
     \begin{subfigure}[b]{0.42\textwidth}
         \centering
         \includegraphics[width=\textwidth]{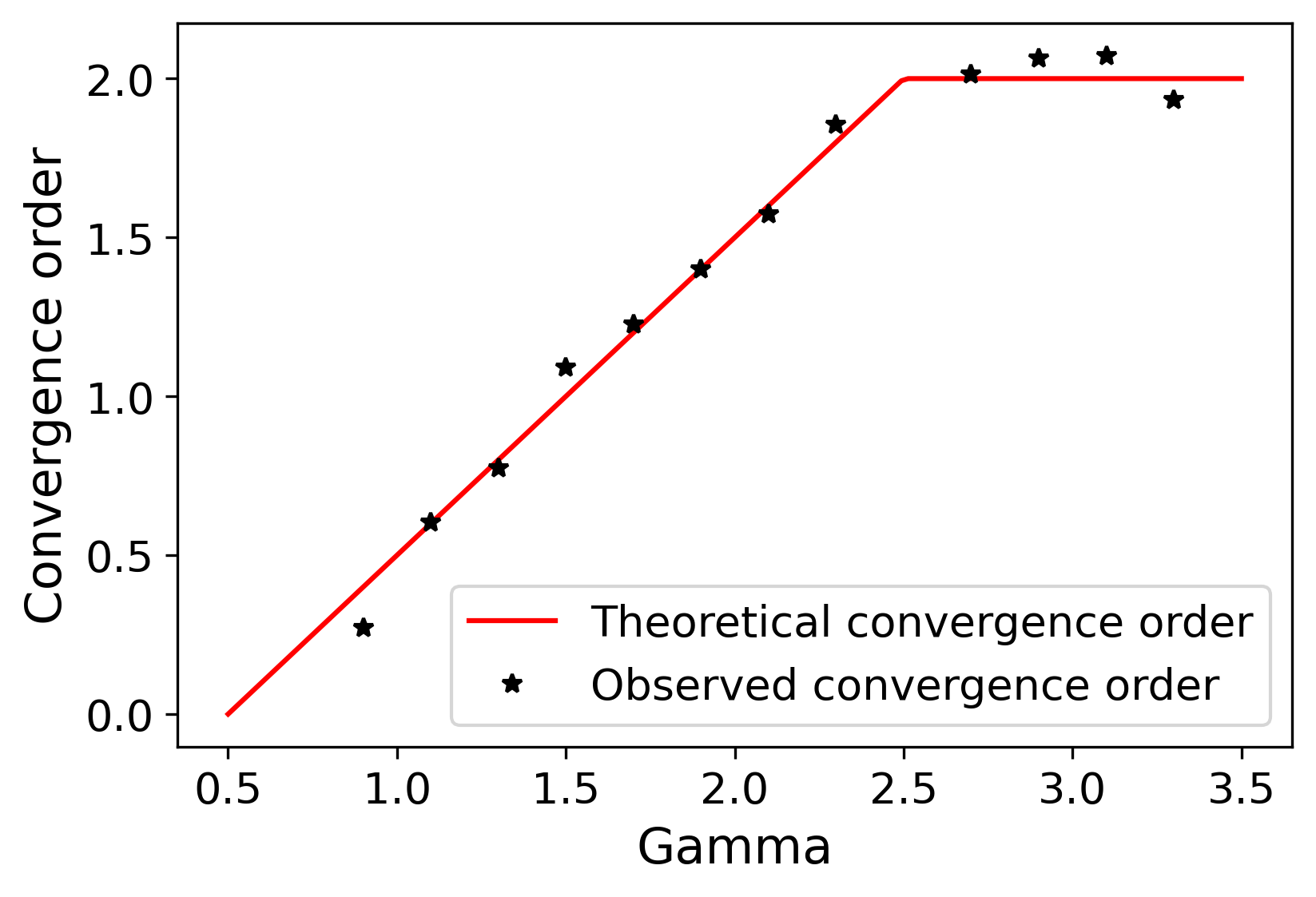}
         \label{fig:constant}
     \end{subfigure}
     \\
     \centering
     \caption{Left: Convergence plots for the projection method for a selection of the values of $\gamma$. The regression line is a linear regression of all the observed errors. The theoretical order line emanates from the left-most point and has slope equal to the theoretical convergence order. Both axes are in $\log_{10}$-scale. Right: Plot of estimated convergence orders for the projection method. The slope of the linear regressions of observed errors for all the tested values of $\gamma$ can be seen in black. The theoretical convergence order is also plotted..}
     \label{fig:projection_convergence_plots}
\end{figure}

Since we are considering only temporal convergence, we restrict ourselves to considering only a single coefficient process $c_k(t)$. Note that since we are considering only a single process the value of $k$ is irrelevant. Our goal is to approximate the relative error
$$\sqrt{\frac{\EV{\abs{c_k(T) - \Tilde{c}^{(m)}_k(T)}^2}}{\EV{\abs{c_k(T)}^2}}} \, ,$$
for different mesh sizes and for different levels of spatial smoothness. We don't know the exact solution $c_k(T)$, so we instead compute a reference solution on a very fine mesh. Once we have a realisation of the noise integrals $\int_{t_{\ell - 1}}^{t_\ell} \left( \frac{s}{t_\ell} \right)^j \e^{-\mu_k (t_\ell - s)} \, \text{d} \omega_k(s)$ we can restrict the noise to a coarser mesh by iteratively using the following trick:
\begin{align*}
& \quad \, \int_{t_{\ell - 2}}^{t_{\ell}} \left( \frac{s}{t_\ell} \right)^j \e^{-\mu_k (t_\ell-s)} \, \text{d} \omega_k(s) \\ & = \e^{-\mu_k h} \left( \frac{t_{\ell - 1}}{t_\ell} \right)^j \int_{t_{\ell - 2}}^{t_{\ell - 1}} \left( \frac{s}{t_{\ell - 1}} \right)^j \e^{-\mu_k (t_{\ell - 1} - s)} \, \text{d} \omega_k(s) \\
& + \int_{t_{\ell - 1}}^{t_\ell} \left( \frac{s}{t_\ell} \right)^j \e^{-\mu_k (t_\ell - s)} \, \text{d} \omega_k(s) \, .
\end{align*}
Once we have the reference solution we can therefore compute $\Tilde{c}^{(m)}_k(t_n)$ on a series of coarser meshes. To get the \textit{mean} square error, we do multiple simulations and compute the sample average of the observed relative error.

We first consider the left-point method. We chose $T = 1$, $\mu_k = 0.1$, and $\lambda_k = 1$. The experiment was run for $\gamma = 0.7, 0.8, 0.9, 1.1, 1.2, 1.3, 1.4, 1.6, 1.7, 1.8, 1.9, 2.0$, and for each $\gamma$ a sample average of $100$ simulations was used to estimate the relative mean square error. We skip $\gamma = 1.0$, since our method is exact in this case, and we skip $\gamma = 1.5$ since the logarithmic term $\delta_{\gamma, h}$ in Corollary \ref{cor:order-0-method} increases the difficulty of estimating the convergence rate. We used a reference mesh with resolution $h = 2^{-17}$ and the rougher meshes have resolutions $2^{-k}$ for $k = 7,8,9,10$. Convergence plots for a selection of the tested $\gamma$'s can be seen in Figure \ref{fig:constant_convergence_plots}. We see that the log-errors are approximately linear in all cases, indicating that we have entered the asymptotic regime. This justifies using a linear regression across the observed errors to approximate the order of convergence; the corresponding regression line has also been plotted in Figure \ref{fig:constant_convergence_plots}. A compilation of the observed regression slopes can be seen in Figure \ref{fig:constant_convergence_plots} alongside the theoretical convergence orders.  

We also test the convergence of the projection method. We again choose $T = 1$, $\mu_k = 1$, and $\lambda_k = 1$. We test $\gamma = 0.9, 1.1, 1.3, 1.5, 1.7, 1.9, 2.1, 2.3, 2.7, 2.9, 3.1, 3.3$, and do $100$ simulations for each $\gamma$. The reference mesh was given resolution $h = 2^{-14}$ and the rougher meshes were given resolutions $2^{-k}$ given by $k = 4,5,6,7$. We limited ourselves to these mesh sizes to avoid issues with machine precision. Convergence plots for a selection of the tested $\gamma$'s can be seen in Figure \ref{fig:constant_convergence_plots}. As in the example with the left-point method the order of convergence was approximated by doing a linear regression across the observed errors; the corresponding regression line can also be seen in Figure \ref{fig:projection_convergence_plots}. As in the previous example, a compilation of the regression slopes can be seen in Figure \ref{fig:projection_convergence_plots} alongside the theoretical orders. 

We see that the observed convergence order well approximate the theoretical convergence order for both the left-point method and the projection method. In Figure \ref{fig:constant_convergence_plots} we see that for the left-point method there is some discrepancy close to the critical point of $\gamma = 1.5$, but it is to be expected that we would have less accuracy close to $\gamma = 1.5$. Indeed an inspection of the {proof of Theorem \ref{thm-convergence}} shows that the constant $C(C_\Pi, t, \gamma)$ blows up as $\gamma$ approaches $m + \frac{3}{2}$. This discrepancy is not visible in Figure \ref{fig:projection_convergence_plots} for the projection method, but this is likely because we tested a sparser selection of $\gamma's$ in this case.

\begin{figure}[p]
     \begin{subfigure}[b]{0.42\textwidth}
         \centering
         \includegraphics[width=\textwidth]{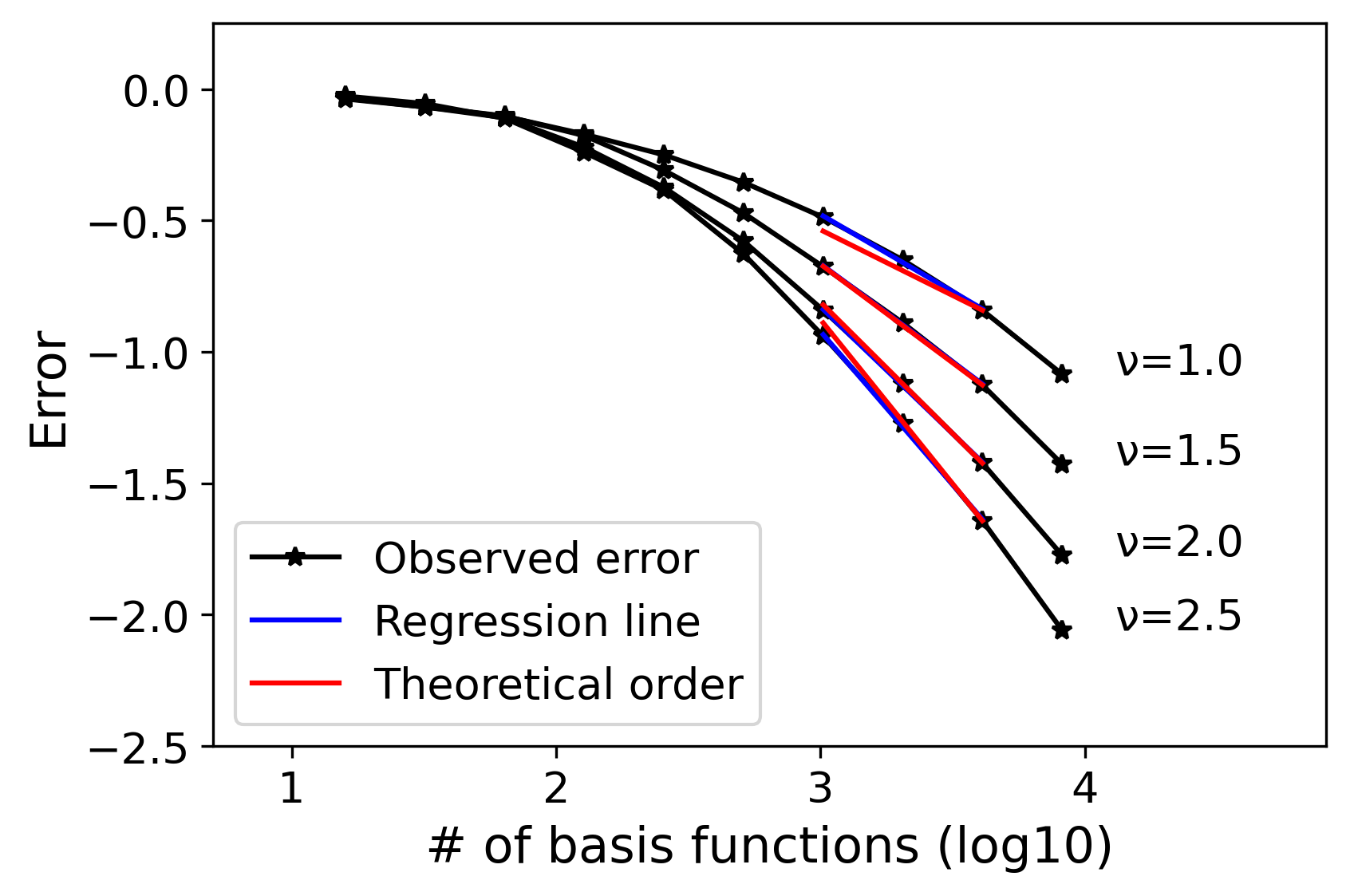}
         \label{fig:constant}
     \end{subfigure}
     \begin{subfigure}[b]{0.42\textwidth}
         \centering
         \includegraphics[width=\textwidth]{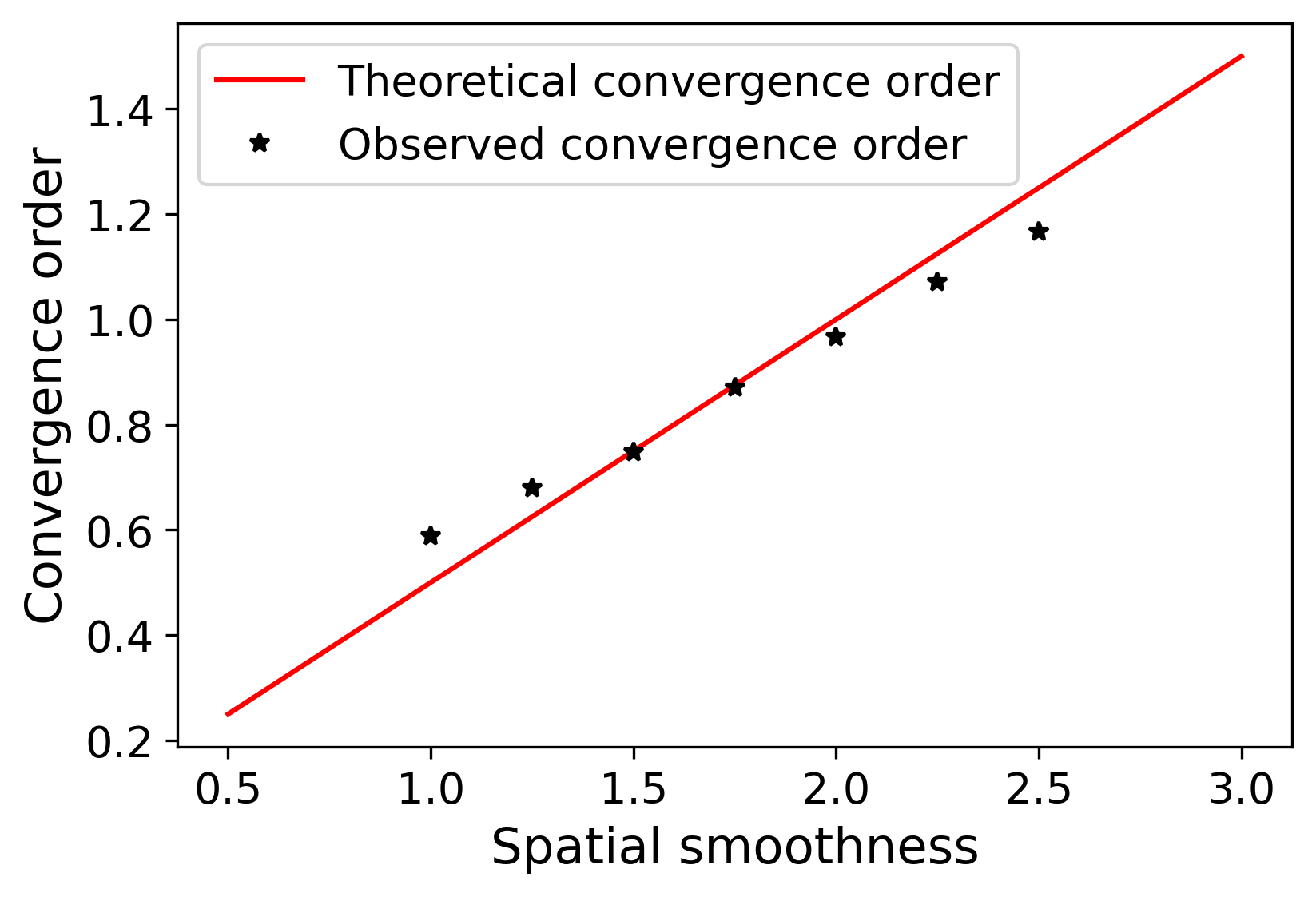}
         \label{fig:constant}
     \end{subfigure}
     \\
     \centering
     \caption{Left: Convergence plots for the relative spectral mean square error for a selection of the values of $\nu_s$. The other parameters have been fixed. The regression line is a linear regression of four of the observed errors. The theoretical order line emanates from the second to right-most point and has slope equal to the theoretical convergence order. Both axes are in $\log_{10}$-scale. Right: Plot of estimated spectral convergence orders. The slope of the linear regressions of four of the observed errors for all the tested values of $\nu_s$ can be seen in black. The other parameters have been fixed. The theoretical convergence order is also plotted.}
     \label{fig:spectral_convergence_plots}
\end{figure}

\begin{figure}[p!]
     \begin{subfigure}[p]{0.43\textwidth}
         \includegraphics[width=\textwidth]{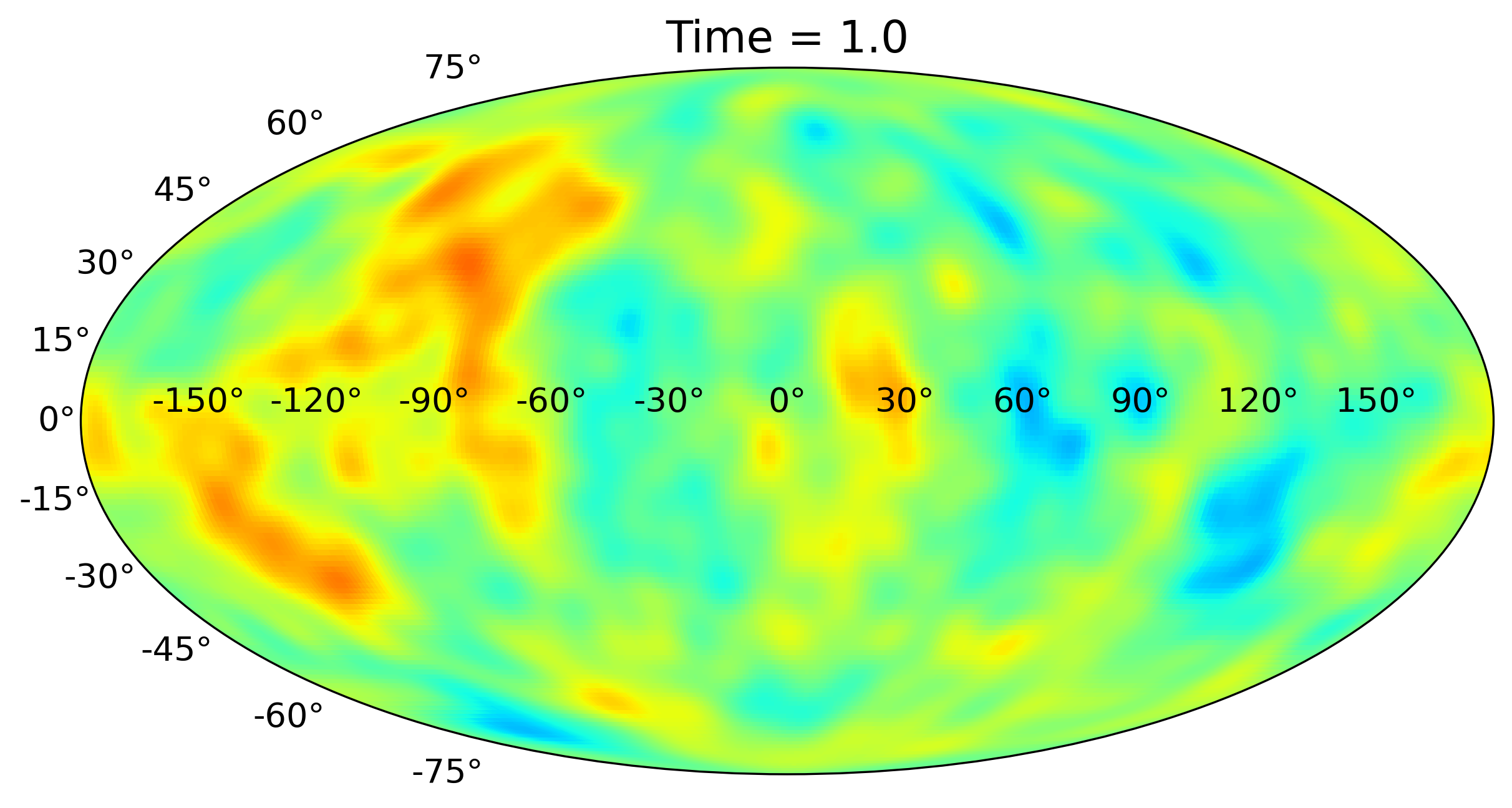}
     \end{subfigure}
     \begin{subfigure}[p]{0.43\textwidth}
         \includegraphics[width=\textwidth]{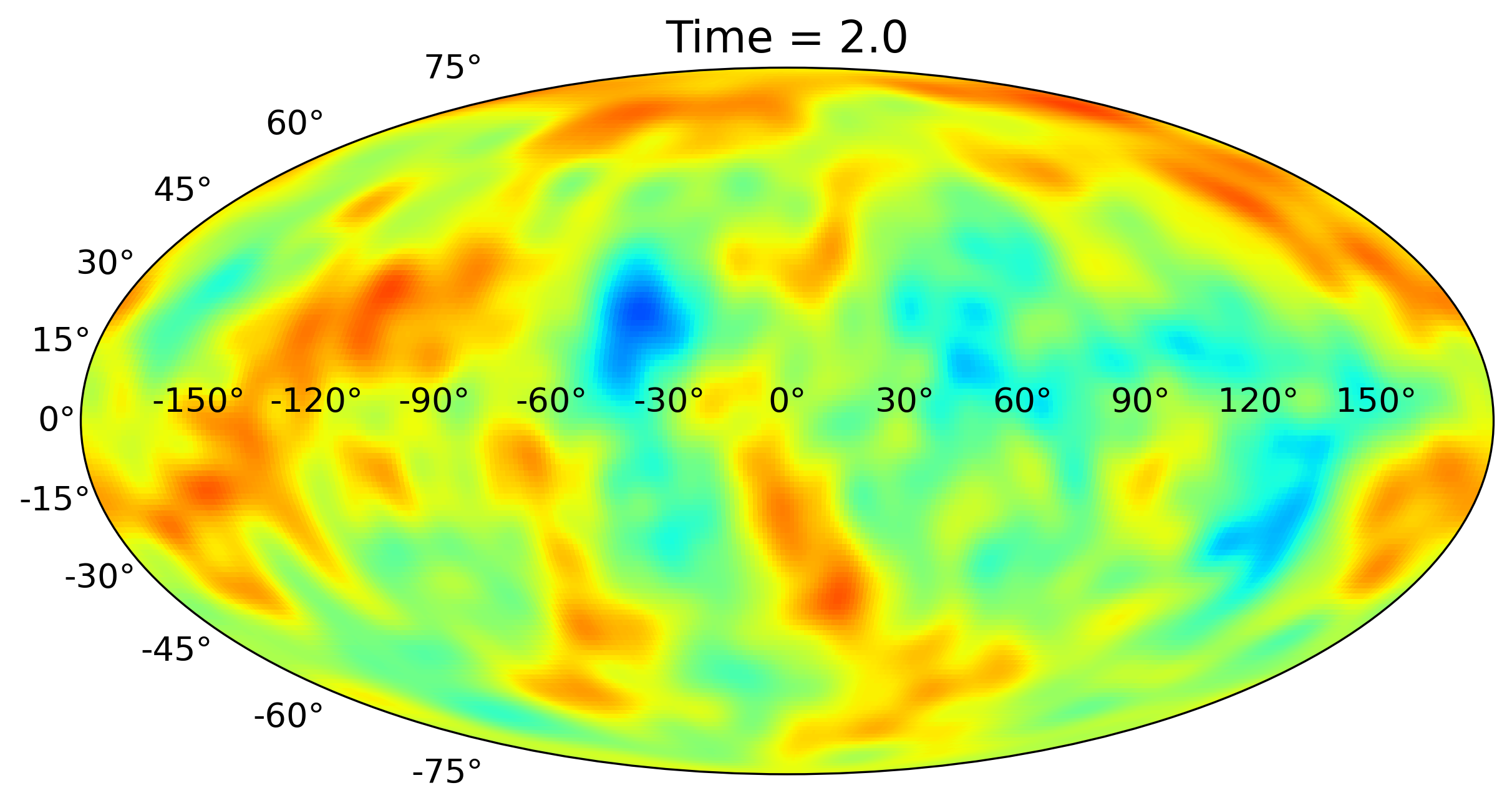}
     \end{subfigure}
     \\
     \hspace*{0.06\textwidth}
     \begin{subfigure}[b]{0.43\textwidth}
         \includegraphics[width=\textwidth]{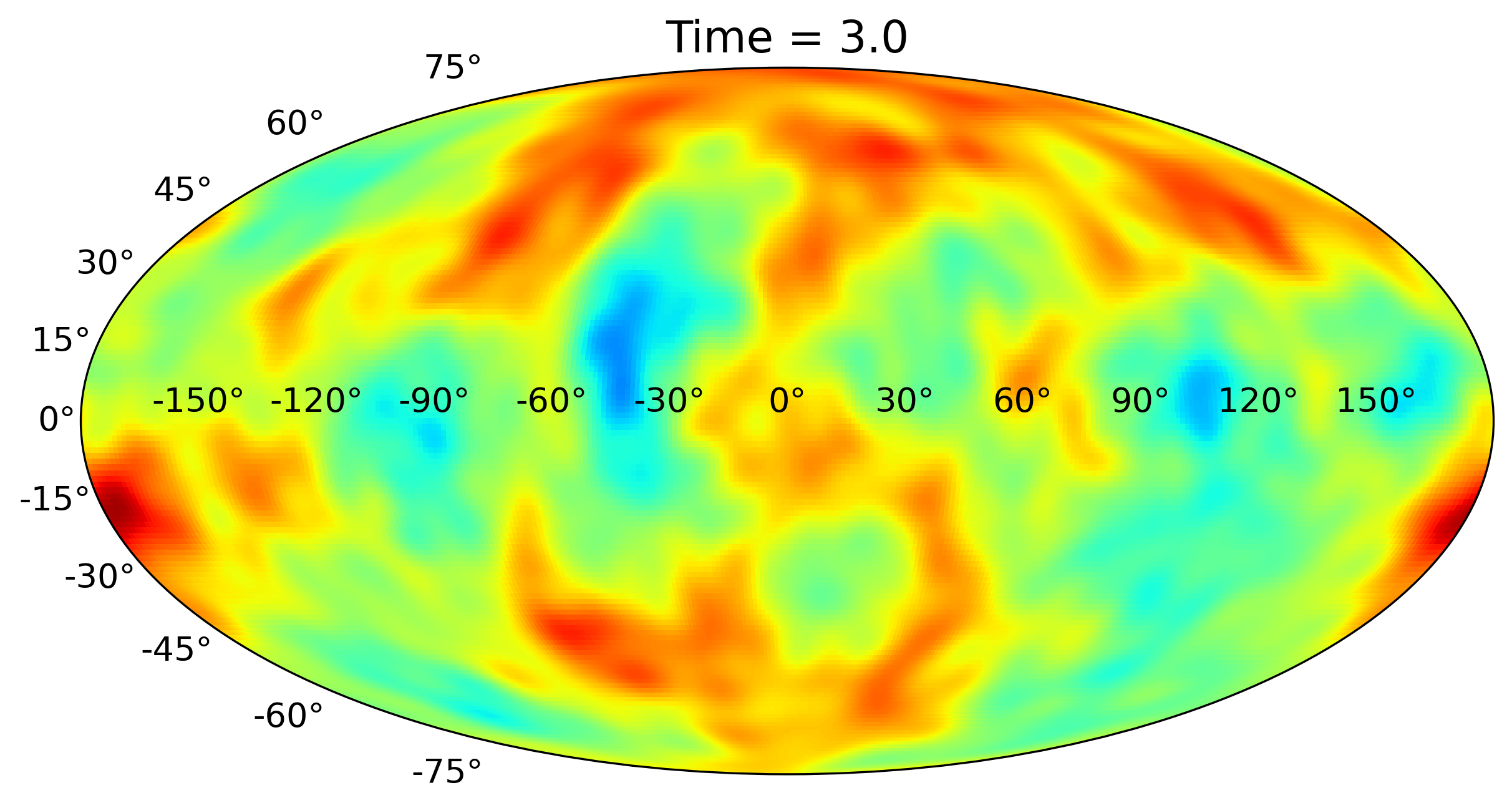}
     \end{subfigure}
     \begin{subfigure}[b]{0.49\textwidth}
         \includegraphics[width=\textwidth]{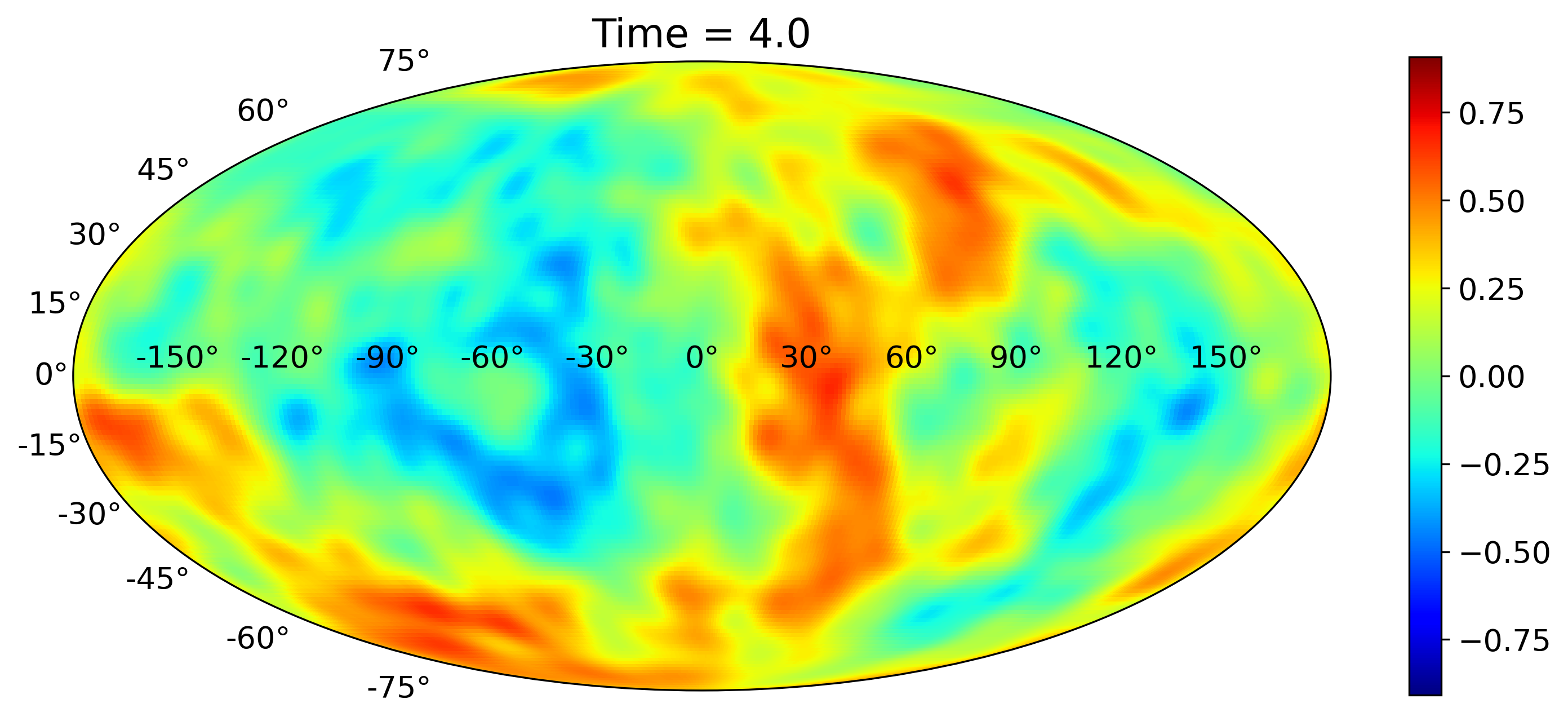}
     \end{subfigure}
     \\
     \caption{A simulation of (\ref{eq:stochastic-convolution}) with the operators in (\ref{eq:operators}) on the unit sphere with $\alpha = 0.5$, $\beta = 1$, $\gamma = 1.5$, $\kappa = 2.828$, $r = 10.0$, and $\sigma = 10$. The plots use a Mollweide projection.}\label{fig:sphere_sim}
     \vspace{1.3cm}
\end{figure}

\begin{figure}[p]
    \raisebox{1.4\height}{
        \centering
        \begin{tabular}{||c|c||}
            \hline
            $\gamma$ & $1.250$ \\
            $\alpha$ & $0.500$ \\
            $\beta$ & $0.750$ \\
            $r$ & $0.816$ \\
            $\kappa$ & $2.000$ \\
            \hline
        \end{tabular}
    }
     \begin{subfigure}[b]{0.3\textwidth}
         \centering
         \includegraphics[width=\textwidth]{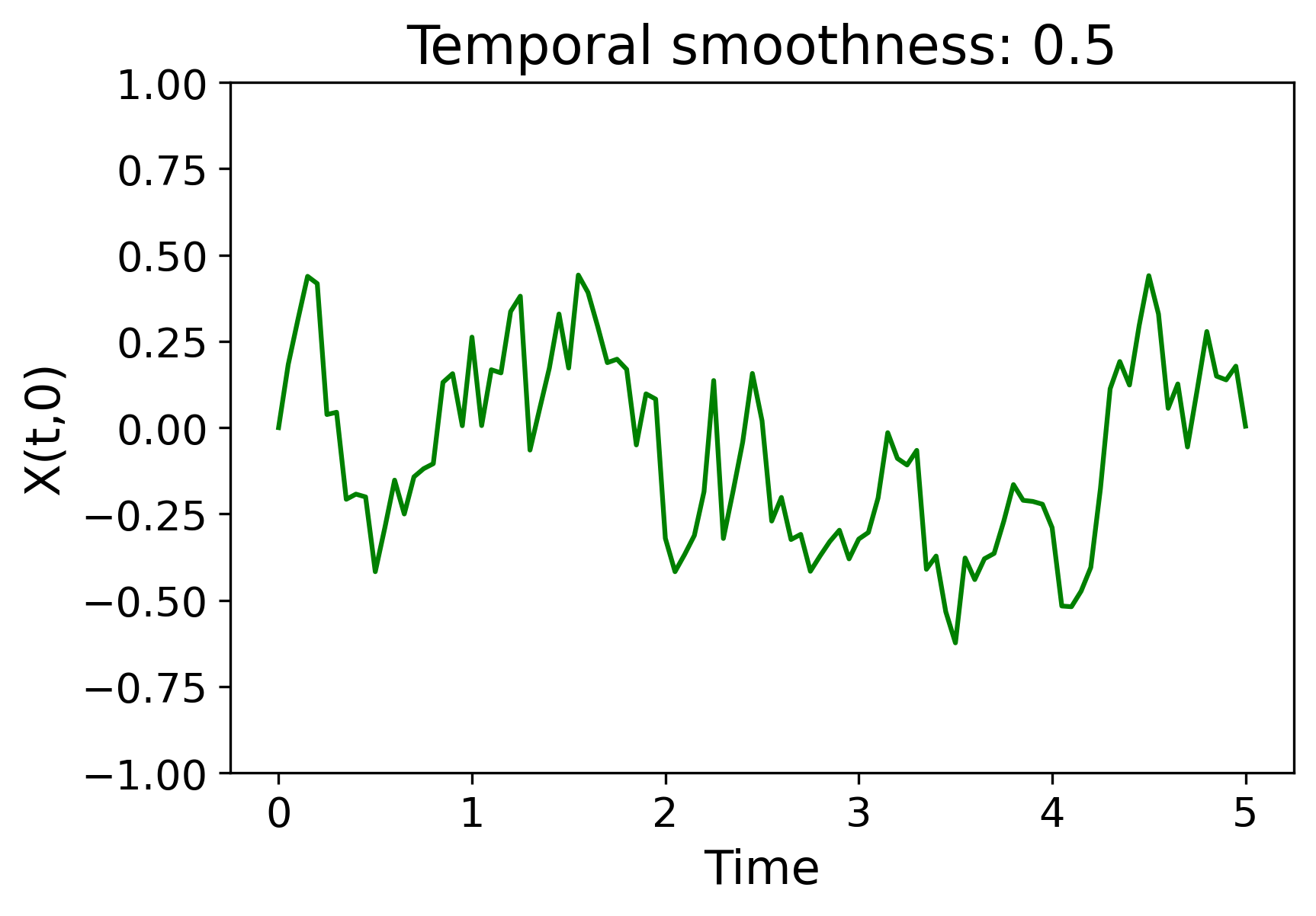}
     \end{subfigure}
     \begin{subfigure}[b]{0.3\textwidth}
         \centering
         \includegraphics[width=\textwidth]{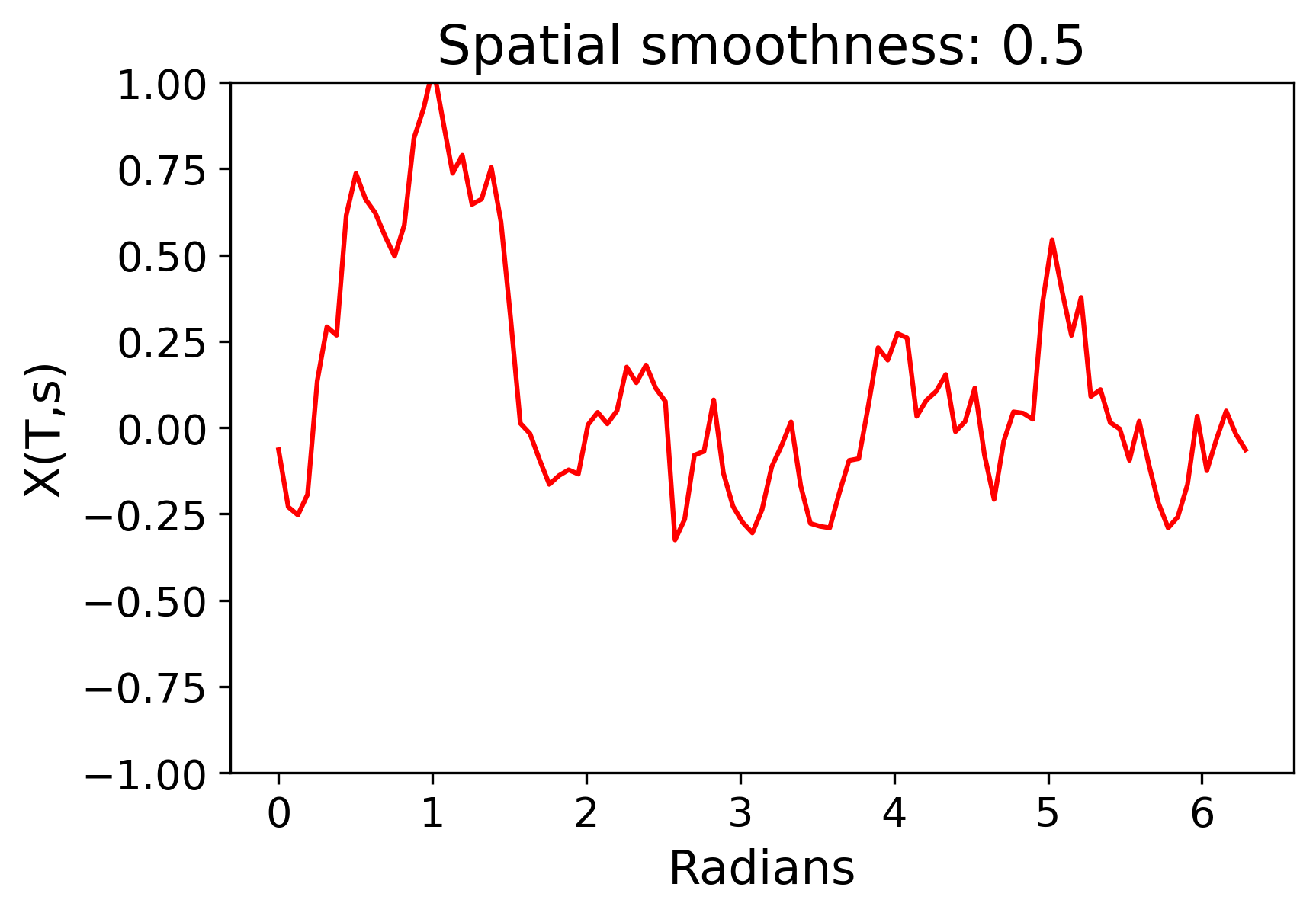}
     \end{subfigure}\\
    \raisebox{1.4\height}{
        \centering
        \begin{tabular}{||c|c||}
            \hline
            $\gamma$ & $2.750$ \\
            $\alpha$ & $0.167$ \\
            $\beta$ & $0.750$ \\
            $r$ & $0.297$ \\
            $\kappa$ & $2.000$ \\
            \hline
        \end{tabular}
    }
     \begin{subfigure}[b]{0.3\textwidth}
         \centering
         \includegraphics[width=\textwidth]{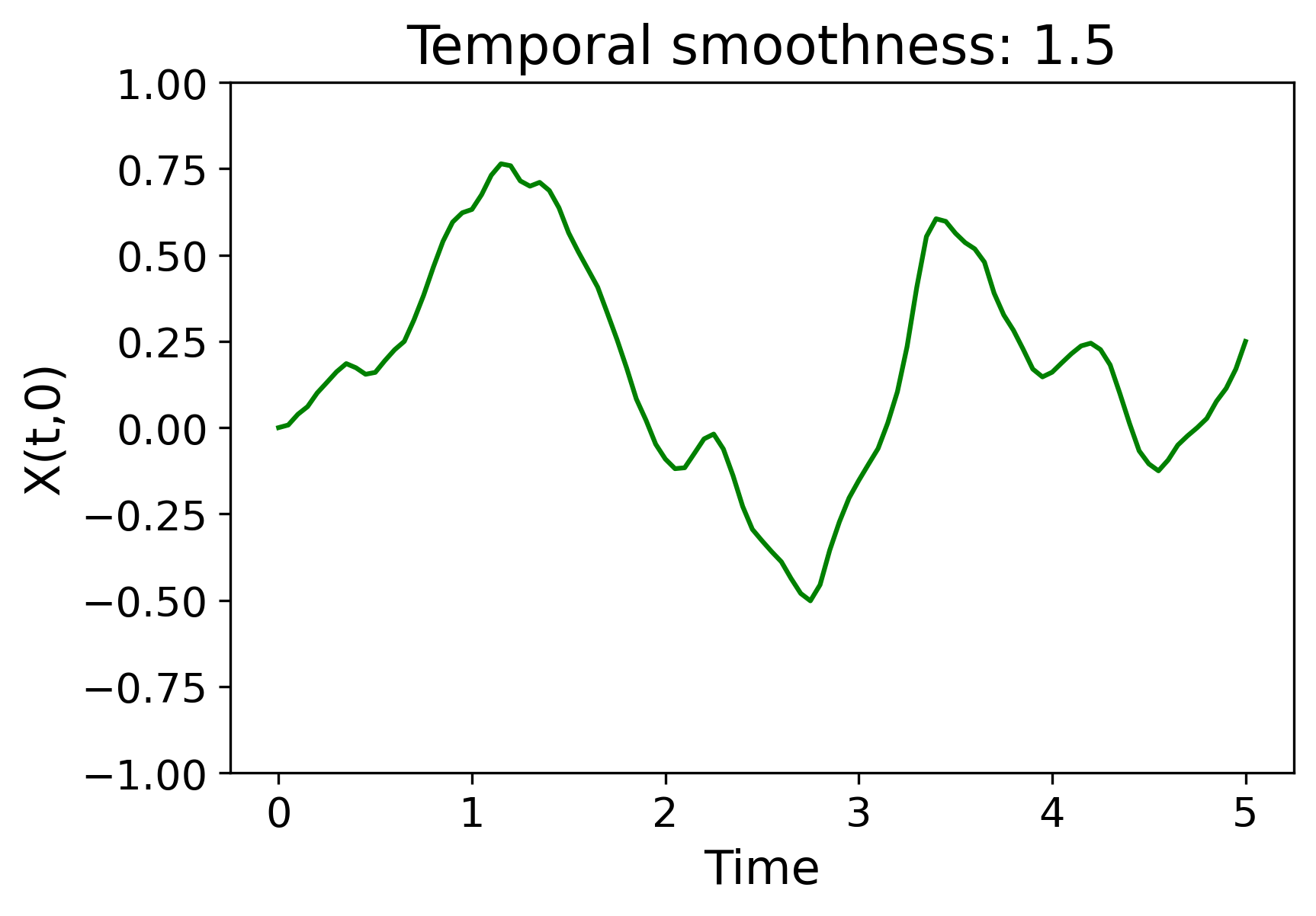}
     \end{subfigure}
     \begin{subfigure}[b]{0.3\textwidth}
         \centering
         \includegraphics[width=\textwidth]{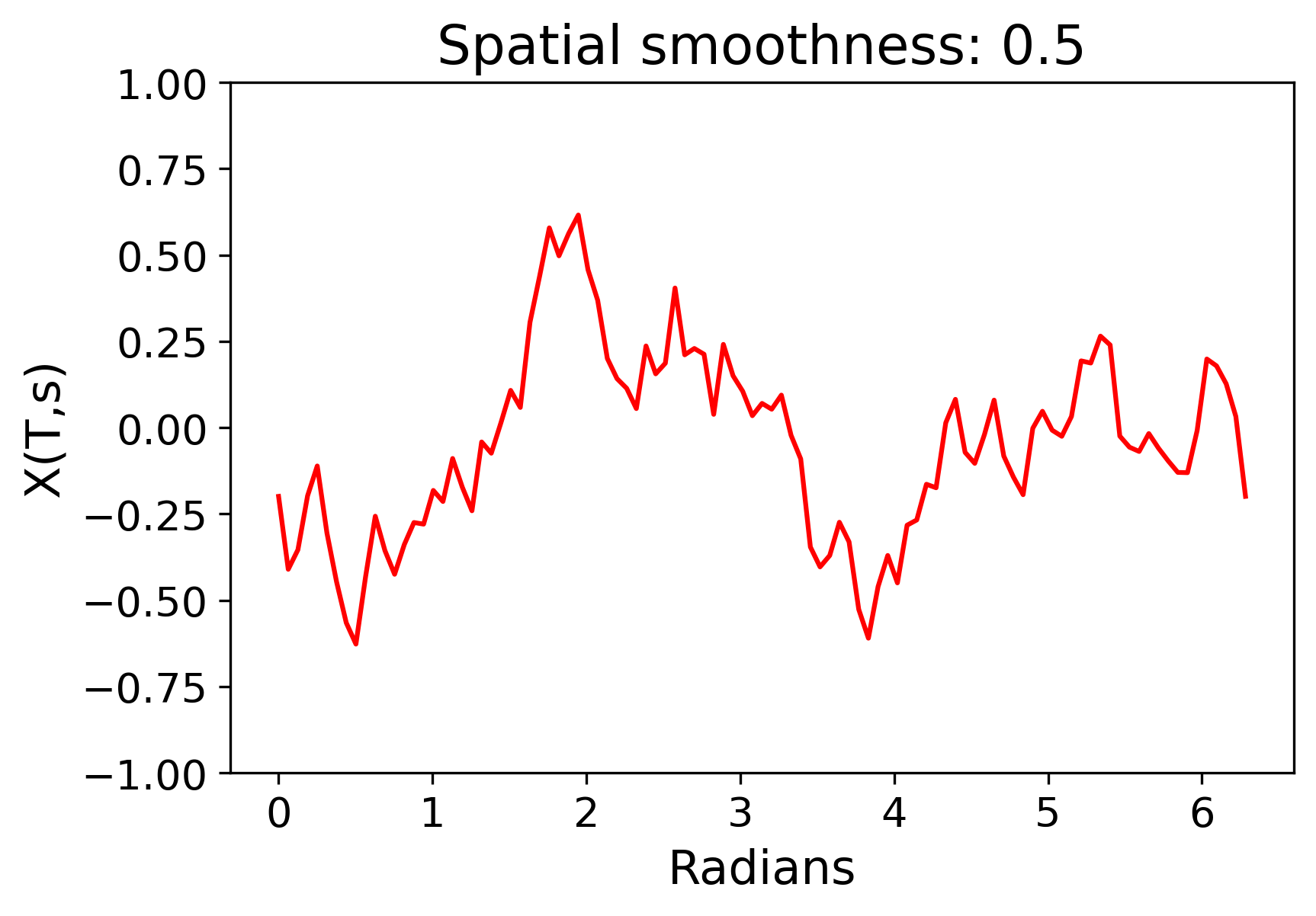}
     \end{subfigure}\\
    \raisebox{1.4\height}{
        \centering
        \begin{tabular}{||c|c||}
            \hline
            $\gamma$ & $2.000$ \\
            $\alpha$ & $0.417$ \\
            $\beta$ & $1.250$ \\
            $r$ & $0.813$ \\
            $\kappa$ & $3.464$ \\
            \hline
        \end{tabular}
    }
     \begin{subfigure}[b]{0.3\textwidth}
         \centering
         \includegraphics[width=\textwidth]{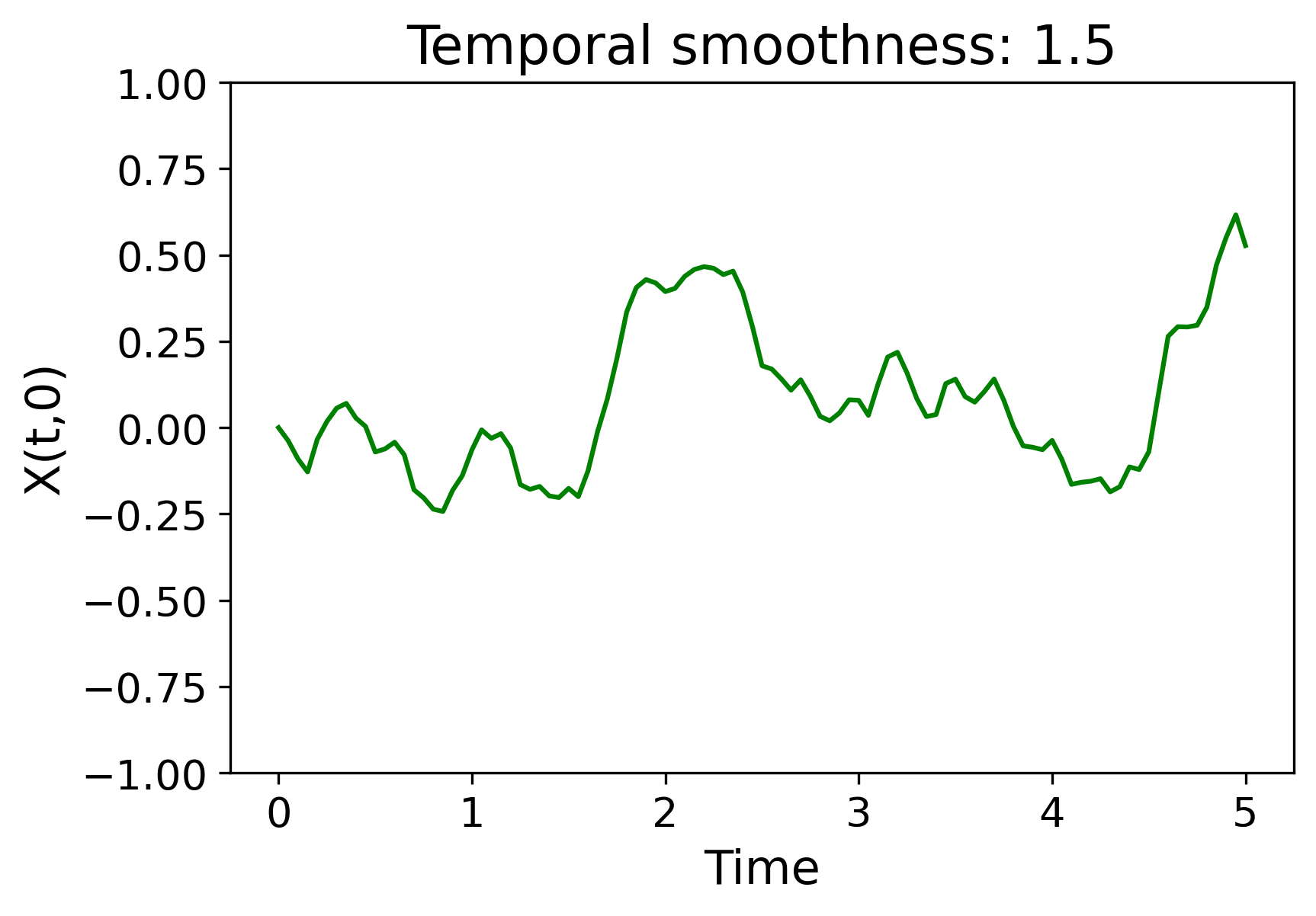}
     \end{subfigure}
     \begin{subfigure}[b]{0.3\textwidth}
         \centering
         \includegraphics[width=\textwidth]{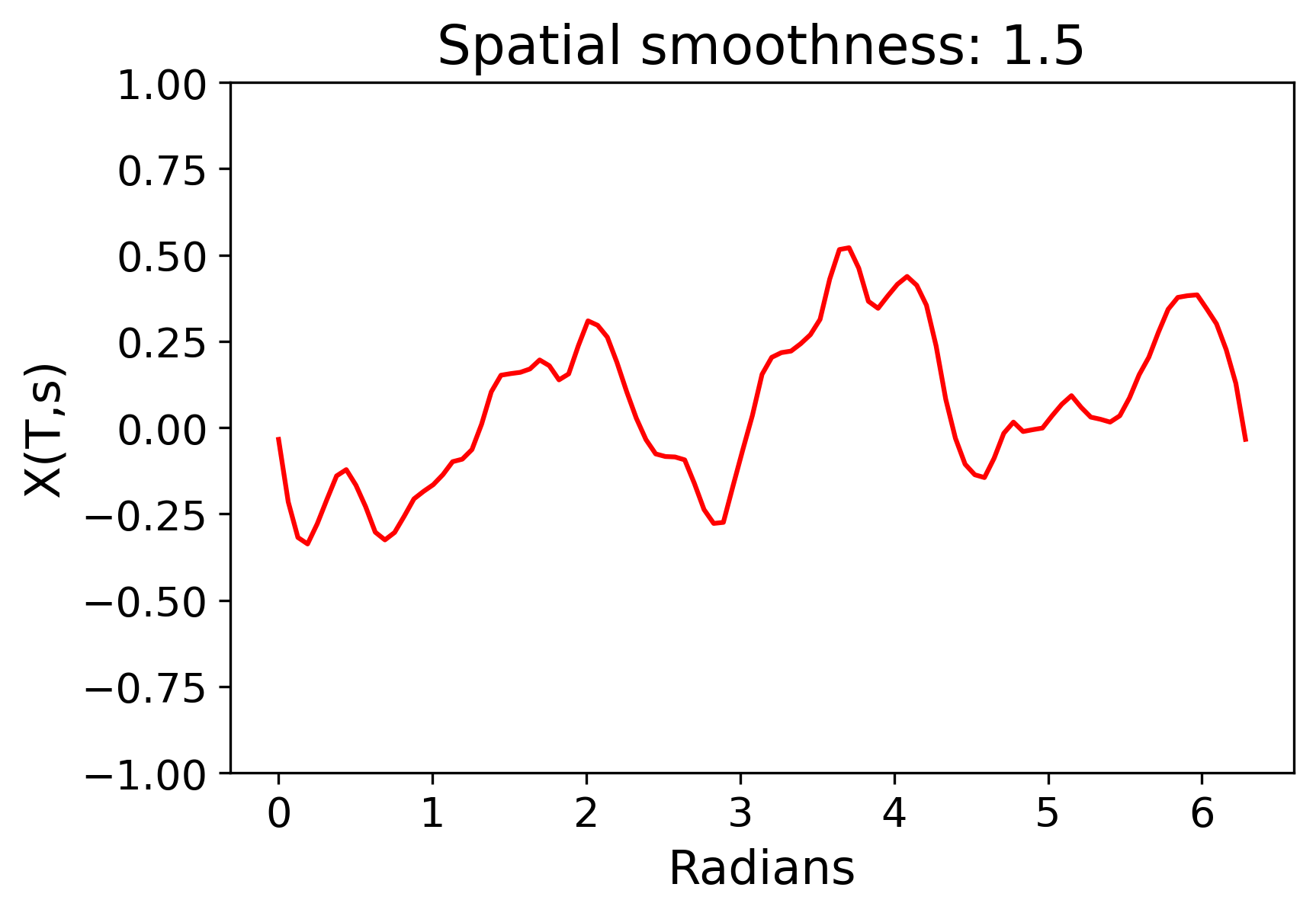}
     \end{subfigure}\\
    \raisebox{1.4\height}{
        \centering
        \begin{tabular}{||c|c||}
            \hline
            $\gamma$ & $3.000$ \\
            $\alpha$ & $0.250$ \\
            $\beta$ & $1.250$ \\
            $r$ & $0.416$ \\
            $\kappa$ & $3.464$ \\
            \hline
        \end{tabular}
    }
     \begin{subfigure}[b]{0.3\textwidth}
         \centering
         \includegraphics[width=\textwidth]{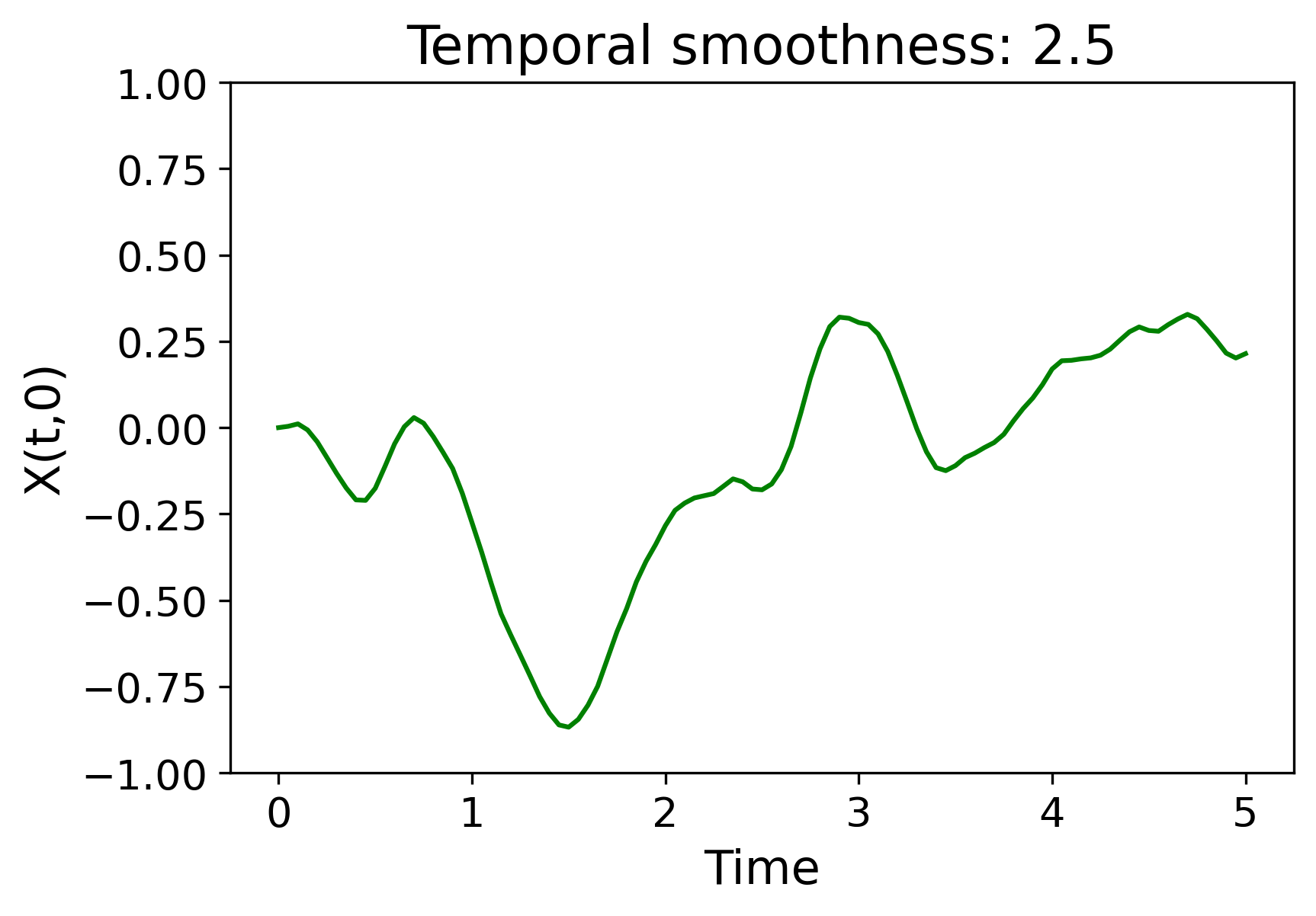}
     \end{subfigure}
     \begin{subfigure}[b]{0.3\textwidth}
         \centering
         \includegraphics[width=\textwidth]{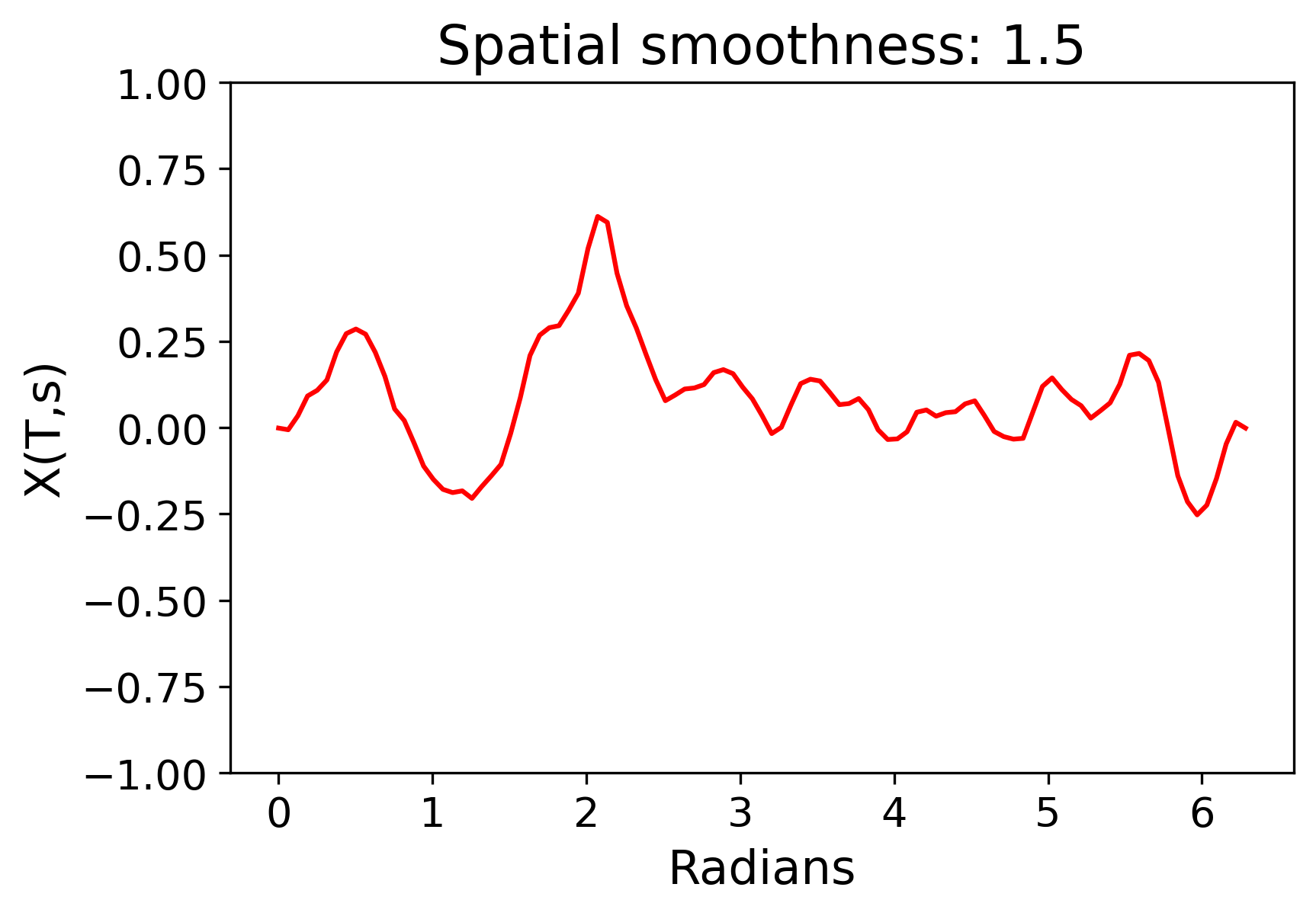}
     \end{subfigure}\\
     \caption{Realisations of temporal and spatial traces for a selection of parameter values. Each row is from the same simulation, containing a list of the chosen parameter values (left), a plot of the temporal trace (middle), and a plot of the spatial trace (right).}
     \label{fig:smoothness-plots}
\end{figure}

\subsection{Testing the convergence of the spectral truncation} 

We also test the spectral convergence rate in Corollary \ref{cor:laplacian-spectral-convergence}. This is computationally more challenging than the previous two examples, since it requires us to also do a temporal discretisation with $h$ small enough so that the temporal error does not interfere with our approximation of the spectral error, as per the discussion in Section \ref{sec:main-result}. To minimise this issue we use the projection method with $h = 2^{-7}$ as our method of discretisation. The temporal domain is $[0,10]$ and errors are evaluated at $T = 1$. The spatial domain is $[0,1]^2$ and we use $A = r^{-1}(\kappa^2 - \Delta)^{\alpha}$ and $Q = \sigma^2 r^{-2\gamma}(\kappa^2 - \Delta)^{-\beta}$. We fix {$\nu_t = 1.0$}, $r_t = 5.0$, $r_s = 0.1$, $\beta_s = 0.5$, and vary $\nu_s = 1.00, 1.25, 1.50, 1.75, 2.00, 2.25, 2.50$. According to Corollary \ref{cor:laplacian-spectral-convergence} the convergence order of the spectral method should then be {$\frac{\nu_s}{2}$}. $2^{14}$ spectral basis functions were used to simulate a reference solution. The relative $\mathcal{L}^2$-error when using $2^L$ basis functions for $L = 1,2,...,13$ was then approximated by comparison to the reference solution. We have only done one simulation per $\nu_s$, as this proved to be sufficient. In Figure \ref{fig:spectral_convergence_plots} we can see convergence plots for a selection of these values of $\nu_s$; the spectral convergence order has been estimated by doing a linear regression of four of the observed errors. We have excluded the last error from the regression because in cases with low $\nu_s$, $L = 13$ is not much more accurate than $L=14$, which makes the computed error dubious. We see that the linear regression over these four points seems to approximate the convergence rate well in most cases, with the possible exception of $\nu_s = 2.5$, where the last four points might have been a better choice. A plot of the linear regression slopes can be seen in Figure \ref{fig:spectral_convergence_plots} and we see that the observed convergence orders well approximate the one predicted by Corollary \ref{cor:laplacian-spectral-convergence}.

\subsection{Simulations on the unit sphere} 

We have also applied our method to simulate a realization of (\ref{eq:fractional-time-heat}) on the unit sphere. We use $[0,5]$ as temporal domain, $1024$ spatial basis functions, and the projection method with $h = 0.1$, to simulate the spectral coefficient processes. We have used $A = r^{-1}(\kappa^2 - \Delta)^{\alpha}$ and $Q = \sigma^2 r^{-2\gamma}(\kappa^2 - \Delta)^{-\beta}$, with $\alpha = 0.5$, $\beta = 1$, $\gamma = 1.5$, $\kappa = 2.828$, $r = 10.0$, and $\sigma = 10$. This gives $\nu_t = 1$, $\nu_s = 1$, $r_t = 5.0$, $r_s = 1.0$, and $\beta_s = 0.5$. A plot of the realisation for half-integer values of $t$ using a Mollweide projection can be seen in Figure \ref{fig:sphere_sim}. To demonstrate our ability to separately control the spatial and temporal smoothness we also simulate temporal and spatial traces on the sphere for a selection of parameters with $4096$ spatial basis functions. The simulated traces can be seen in Figure \ref{fig:smoothness-plots}. The temporal trace is the trace at one point on the sphere, while the spatial trace is the trace along a great circle on the sphere at $t = 5$. While it is impossible to directly assess the smoothness of a function by computing only a finite number of point values, {a visual inspection of the graphs of the functions} gives us an informal indication that the curve becomes more regular as we increase the corresponding parameters.

\newpage

\section{Proofs} \label{sec:proofs}

{
\begin{proof}[Proof of Theorem \ref{thm:spectral-convergence}] Let $N > M$. Clearly $\Tilde{X}_N(t) - \Tilde{X}_M(t) = \sum_{k = M + 1}^{N} c_k(t) e_k$. Using Lemma \ref{lemma:coefficient-representation}, the Itô isometry, the definition of the gamma function, Assumption \ref{assumption:operator}, and $1 + r_\lambda - (2 \gamma-1) r_\mu < 0$, we get 
\begin{align*}
    \EV{\norm{\Tilde{X}_N(t) - \Tilde{X}_M(t)}^2} &= \EV{ \left\| \sum_{k = M + 1}^{N} c_k(t) e_k \right\|^2 } \\
    & = \EV{ \sum_{k = M + 1}^{N} \left\|\frac{\sqrt{\lambda_k}}{\Gamma(\gamma)} \int_0^t \e^{-\mu_k (t-s)} (t-s)^{\gamma - 1} \text{d} \omega_k(s) e_k \right\|^2 } \\
    &= \frac{1}{\Gamma(\gamma)^2} \sum_{k = M + 1}^{N} \lambda_k \int_0^t \e^{-2 \mu_k (t-s)} (t-s)^{2 \gamma - 2} \, \dint s \\
    &= \frac{1}{\Gamma(\gamma)^2} \sum_{k = M + 1}^{N}  \frac{\lambda_k}{(2 \mu_k)^{2 \gamma - 1}} \int_0^{2 \mu_k t} \e^{-u} u^{2 \gamma - 2} \, \dint u \\
    &\leq \frac{\Gamma(2 \gamma - 1)}{\Gamma(\gamma)^2} \sum_{k = M + 1}^{N} \lambda_k \mu_k^{1 - 2 \gamma} \\
    &\leq \frac{C_2}{C_1^{2\gamma-1}} \frac{\Gamma(2 \gamma - 1)}{\Gamma(\gamma)^2} \sum_{k = M + 1}^{N} k^{r_\lambda - (2 \gamma-1) r_\mu} \\
    &\leq  \frac{C_2 C_1^{1 - 2\gamma} \Gamma(2 \gamma - 1)}{\Gamma(\gamma)^2 ((2 \gamma-1) r_\mu - r_\lambda - 1)} M^{1 + r_\lambda - (2 \gamma-1) r_\mu} \, ,
\end{align*}
so $\{X_M(t)\}_{M = 1}^\infty$ is a Cauchy sequence in $\mathcal{L}^2(\Omega, \mathcal{L}^2(\mathcal{D}))$. It follows that $X(t) = \sum_{k = 1}^\infty c_k(t) e_k \in \mathcal{L}^2(\Omega, \mathcal{L}^2(\mathcal{D}))$. By taking the limit as $N \rightarrow \infty$ we obtain
\begin{align*}
    \EV{\norm{X(t) - \Tilde{X}_M(t)}^2} & \leq  \frac{C_2 C_1^{1 - 2\gamma} \Gamma(2 \gamma - 1)}{\Gamma(\gamma)^2 ((2 \gamma-1) r_\mu - r_\lambda - 1)} M^{1 + r_\lambda - (2 \gamma-1) r_\mu} \, . \qedhere
\end{align*} 
\end{proof}
}

To prove Theorem \ref{thm-convergence} we first need a preliminary lemma.

\begin{lem} \label{lemma:function_approximation} Let $t \in [0,T]$, $\mu_k > 0$, $\gamma > \frac{1}{2}$, and ${\rho} \in \mathcal{L}^2([0,t))$. Let $\{\omega_k(s)\}_{s \in [0,T]}$ be an $\mathbb{R}$-valued, standard Wiener-process. Then
$$\EV{\abs{ \int_0^t \e^{- \mu_k (t-s)} ((t-s)^{\gamma - 1} - {\rho}(s)) \text{d} \omega_k(s)}^2} \leq \norm{(t - s)^{\gamma - 1} - {\rho}(s)}_{\mathcal{L}^2([0,t))}^2 \, .$$
In addition for $p \in (2, \infty]$ we have that 
\begin{align*}
\EV{\abs{ \int_0^t \e^{- \mu_k (t-s)} ((t-s)^{\gamma - 1} - {\rho}(s)) \text{d} \omega_k(s)}^2} \\ \leq \left( q \mu_k \right)^{-\frac{2}{q}} \norm{(t-s)^{\gamma - 1} - {\rho}(s)}_{\mathcal{L}^p([0,t))}^2 \, ,
\end{align*}
where $q \in [2,\infty)$ is such that $\frac{1}{p} + \frac{1}{q} = \frac{1}{2}$.

\end{lem}

\begin{proof}


By the Itô isometry~\cite[Theorem 4.27]{daprato2014}
 \begin{align*} & \left( \EV{\abs{ \int_0^t \e^{- \mu_k (t-s)} ((t-s)^{\gamma - 1} - {\rho}(s)) \text{d} \omega_k(s)}^2} \right)^{\frac{1}{2}} \\
 = & \, \left( \int_0^t \e^{- 2 \mu_k (t-s)} ((t-s)^{\gamma - 1} - {\rho}(s))^2 \text{d} s \right)^{\frac{1}{2}} \\
= & \,  \norm{\e^{-\mu_k (t-s)} ((t-s)^{\gamma - 1} - {\rho}(s))}_{\mathcal{L}^2([0,t))} \\
\leq & \,  \norm{\e^{-\mu_k (t-s)}}_{\mathcal{L}^q([0,t))} \norm{(t-s)^{\gamma - 1} - {\rho}(s)}_{\mathcal{L}^p([0,t))} \, ,
 \end{align*}
where $\frac{1}{q} + \frac{1}{p} = \frac{1}{2}$ for $p,q \in [2, \infty]$. For $q = \infty$ and $p = 2$ we have $\norm{\e^{-\mu_k (t-s)}}_{\mathcal{L}^q([0,t))} = 1$ and thus 
$$\EV{\abs{ \int_0^t \e^{- \mu_k (t-s)} ((t-s)^{\gamma - 1} - {\rho}(s)) \text{d} \omega_k(s)}^2} \leq  \norm{(t-s)^{\gamma - 1} - {\rho}(s)}_{\mathcal{L}^2([0,t))}^2 \, .$$
For $q < \infty$ we calculate
\begin{align*}
    \norm{\e^{-\mu_k (t-s)}}_{\mathcal{L}^q([0,t))} & = \left( \int_0^t \e^{-q \mu_k (t-s)} \, \text{d}s  \right)^{\frac{1}{q}} = \left( \frac{1}{q \mu_k} (1 - \e^{-q \mu_k t}) \right)^{\frac{1}{q}} \leq (q \mu_k)^{-1/q} \, ,
\end{align*}
so that
$$\EV{\abs{ \int_0^t \e^{- \mu_k (t-s)} ((t-s)^{\gamma - 1} - {\rho}(s)) \text{d} \omega_k(s)}^2} \leq \left( q \mu_k \right)^{-\frac{2}{q}} \norm{(t-s)^{\gamma - 1} - {\rho}(s)}_{\mathcal{L}^p([0,t))}^2 \, .$$
\end{proof}

\begin{proof}[Proof of Theorem \ref{thm-convergence}] We divide the proof into four parts. In Part 1 we consider the case where $\gamma > m + \frac{3}{2}$. In Part 2, we prove the second part of the theorem where $\gamma > m + \frac{3}{2} + \frac{1}{q}$. In Part 3 we consider the case $\gamma < m + \frac{3}{2}$. Finally, in Part 4, we consider the case $\gamma = m + \frac{3}{2}$.

\textbf{Part 1:} Assume $\gamma > m + \frac{3}{2}$, then $f_t \big|_{[t_{\ell - 1}, t_\ell] \cap [0,t)} \in \mathcal{H}^{m + 1}([t_{\ell - 1}, t_\ell) \cap [0,t))$, where $\mathcal{H}^{m+1}$ is the $(m+1)$th Sobolev space with $p = 2$. We can then apply the Bramble-Hilbert lemma~\cite[Chapter 3, Remark 2]{strang}, to show that for each $\ell = 1,2,...N$ there exists a polynomial $g_\ell$ of degree $m$ satisfying
$$\norm{f_t - g_\ell}_{\mathcal{L}^2([t_{\ell - 1}, t_\ell])} \leq C \norm{f}_{\mathcal{H}^{m + 1}([t_{\ell - 1}, t_\ell])} (t_\ell - t_{\ell - 1})^{m + 1} \, .$$
Define $g: [0,t) \rightarrow \mathbb{R}$ to be the function that is equal to $g_\ell$ when restricted to $[t_{\ell - 1}, t_{\ell}) \cap [0,t)$. Then $\norm{f_t - g}_{\mathcal{L}^2([0,t))} \leq C \norm{f}_{\mathcal{H}^{m + 1}([0,t))} h^{m + 1}$. Further 
$$\norm{f_t - \Pi_m f_t}_{\mathcal{L}^2([0,t))} \leq \norm{f_t - g}_{\mathcal{L}^2([0,t))} + \norm{g - \Pi_m g}_{\mathcal{L}^2([0,t))} + \norm{\Pi_m (g - f_t) }_{\mathcal{L}^2([0,t))} \, .$$
Since $g \in \mathcal{P}_m(\mathcal{I})$ we have that $\Pi_m g = g$, so that 
\begin{align*}
\norm{f_t - \Pi_m f_t}_{\mathcal{L}^2([0,t))} & \leq (1 + C_\Pi) \norm{f_t - g}_{\mathcal{L}^2([0,t))} \\
& \leq C (1 + C_\Pi)  \norm{f_t}_{\mathcal{H}^{m + 1}([0,t))} h^{m + 1} \, .
\end{align*}
By the first part of Lemma \ref{lemma:function_approximation} 
\begin{align*} 
\left(\EV{\abs{c_k(t) - \Tilde{c}^{(m)}_k(t)}^2} \right)^{\frac{1}{2}} &= \frac{\sqrt{\lambda_k}}{\Gamma(\gamma)} \left(\EV{\abs{\int_0^t \e^{-\mu_k (t-s)} ((t - s)^{\gamma - 1} -  \Pi_m f_t (s)) \, \text{d} \omega_k(s)}}^2\right)^{\frac{1}{2}} \\
& = \frac{\sqrt{\lambda_k}}{\Gamma(\gamma)} \norm{f_t - \Pi_mf_t}_{\mathcal{L}^2([0,t))}  \\
& \leq C \frac{\sqrt{\lambda_k}}{\Gamma(\gamma)} (1 + C_\Pi)  \norm{f_t}_{\mathcal{H}^{m + 1}([0,t))} h^{m + 1}
\end{align*}

\textbf{Part 2:} Assume $\gamma > m + \frac{3}{2} + \frac{1}{q}$, then $\gamma > m + 2 - \frac{1}{p}$, where $\frac{1}{p} + \frac{1}{q} = \frac{1}{2}$. Thus $f_t \big|_{[t_{\ell - 1}, t_\ell] \cap [0,t)} \in \mathcal{W}^{m + 1, p}([t_{\ell - 1}, t_\ell) \cap [0,t))$, where $\mathcal{W}^{m + 1, p}$ is the $(m + 1)$th order Sobolev space, so we can modify the proof by using the Bramble-Hilbert lemma on $\mathcal{L}^{p}([0,t))$ above to get
\begin{align*}
\norm{f_t - \Pi_m f_t}_{\mathcal{L}^p([0,t))} & \leq C (1 + C_\Pi) h^{m + 1} \norm{f_t}_{\mathcal{W}^{m + 1, p}([0,t))} \, .
\end{align*}
Thus by the second part of Lemma \ref{lemma:function_approximation} 
\begin{align*} 
\left(\EV{\abs{c_k(t) - \Tilde{c}^{(m)}_k(t)}^2} \right)^{\frac{1}{2}} & \leq \frac{\sqrt{\lambda_k} \mu_k^{-\frac{1}{q}}}{\Gamma(\gamma)} \norm{f_t -  \Pi_m f_t}_{\mathcal{L}^p([0,t))}  \\
& \leq C \frac{\sqrt{\lambda_k}\mu_k^{-\frac{1}{q}}}{\Gamma(\gamma)} (1 + C_\Pi^p) \norm{f_t}_{\mathcal{W}^{m + 1, p}([0,t))}  h^{m + 1}.
\end{align*}

\textbf{Part 3:} Assume $\gamma < m + \frac{3}{2}$. Let $\epsilon > 0$ and define $T_m(s) := \sum_{k = 0}^{m} \frac1{k!} f_t^{(k)}(t - \epsilon) (s - t + \epsilon)^k$; the Taylor polynomial expansion of $f_t$ around $t - \epsilon$. Further we define $f_t^\epsilon: [0,t) \rightarrow \mathbb{R}$ as 
\begin{equation*}
f_t^\epsilon(s) := \begin{cases}
    (t - s)^{\gamma - 1}, \quad & \text{if} \quad s < t - \epsilon,\\
    T_m(s), \quad & \text{else.}
\end{cases} 
\end{equation*}
We then consider the approximation 
$$\norm{f_t - g}_{\mathcal{L}^2([0,t))} \leq \norm{f_t - f_t^\epsilon}_{\mathcal{L}^2([0,t))} + \norm{f_t^\epsilon - g}_{\mathcal{L}^2([0,t))} = (I) + (II) \, ,$$
for some $g \in \mathcal{P}_m(\mathcal{I})$. For $s \in [t - \epsilon, t)$ the Schlömilch form of the Taylor remainder is $f_t(s) - f_t^\epsilon(s) = f_t(s) - T_m(s) = \frac{f_t^{(m + 1)}(\xi)}{p m!} (s - \xi)^{m + 1 - p} (s - t + \epsilon)^p$, where $\xi \in (t - \epsilon, s)$ and $p > 0$~\cite[p.66]{Beesack1966}. The $(m+1)$th derivative of $f_t$ is $f_t^{(m + 1)}(s) = K_{m + 1} (t - s)^{\gamma - 2 - m}$, where $K_{m + 1} := (-1)^{m + 1} \Pi_{i = 1}^{m + 1}(\gamma - i)$. We now fix $\delta > 0$ and select $p = \gamma - \frac{1}{2} - \frac{\delta}{2}$, where $\delta \in (0, 1)$ is chosen such that $p > 0$ as required. We then obtain the following expression for the error $f_t(s) - f_t^\epsilon(s)$:
\begin{align*}
f_t(s) - f_t^\epsilon(s) & = \frac{K_{m + 1} (t - \xi)^{\gamma - 2 -m}}{p m!} (s - \xi)^{m + 1 - p} (s - t + \epsilon)^p \\
& = \frac{K_{m + 1} (t - \xi)^{\gamma - 2 -m }}{(\gamma - \frac{1}{2} - \frac{\delta}{2}) m!} (s - \xi)^{m + \frac{3}{2} - \gamma + \frac{\delta}{2}} (s - t + \epsilon)^{\gamma - \frac{1}{2} - \frac{\delta}{2}} \\
& \leq \frac{K_{m + 1} (t - \xi)^{\gamma - 2 -m }}{(\gamma - \frac{1}{2} - \frac{\delta}{2}) m!} (t - \xi)^{m + \frac{3}{2} - \gamma + \frac{\delta}{2}} \epsilon^{\gamma - \frac{1}{2} - \frac{\delta}{2}} \\
& = \frac{K_{m + 1} }{(\gamma - \frac{1}{2} - \frac{\delta}{2}) m!} (t - \xi)^{- \frac{1}{2} + \frac{\delta}{2}} \epsilon^{\gamma - \frac{1}{2} - \frac{\delta}{2}} \\
& \leq \frac{K_{m + 1} }{(\gamma - \frac{1}{2} - \frac{\delta}{2}) m!} (t - s)^{- \frac{1}{2} + \frac{\delta}{2}} \epsilon^{\gamma - \frac{1}{2} - \frac{\delta}{2}}
\end{align*}
The bound $(s - t + \epsilon)^{\gamma - \frac{1}{2} - \frac{\delta}{2}} < \epsilon^{\gamma - \frac{1}{2} - \frac{\delta}{2}}$ is justified by $s - t + \epsilon < t - t + \epsilon = \epsilon$ and the fact that $\gamma - \frac{1}{2} - \frac{\delta}{2} = p > 0$. The bound $(s - \xi)^{m + \frac{3}{2} - \gamma + \frac{\delta}{2}} < (t - \xi)^{m + \frac{3}{2} - \gamma + \frac{\delta}{2}}$ is justified by $s - \xi < t - \xi$ and the fact that $m + \frac{3}{2} - \gamma + \frac{\delta}{2} > 0$ since $\gamma - \frac{\delta}{2} < \gamma < m + \frac{3}{2}$ by assumption. Finally $(t - \xi)^{- \frac{1}{2} + \frac{\delta}{2}} < (t - s)^{- \frac{1}{2} + \frac{\delta}{2}}$ is justified by $t - \xi > t - s$ together with the fact that $- \frac{1}{2} + \frac{\delta}{2} < 0$. We are now ready to approximate $(I)$:
\begin{align*}
\norm{f_t - f_t^\epsilon}_{\mathcal{L}^2([0,t))}^2 & \leq \int_{t - \epsilon}^t \left( \frac{K_{m + 1}}{(\gamma - \frac{1}{2} - \frac{\delta}{2}) m!} \right)^2 (t - s)^{-1 + \delta} \epsilon^{2\gamma - 1 - \delta} \, \dint s \\ 
& = \frac{1}{\delta} \left( \frac{K_{m + 1}}{(\gamma - \frac{1}{2} - \frac{\delta}{2}) m!} \right)^2 \epsilon^{2 \gamma - 1} \, .
\end{align*}
If we optimize for the smallest possible constant in terms of $\delta$ we get the bound 
$$(I) = \norm{f_t - f_t^\epsilon}_{\mathcal{L}^2([0,t))} \leq  \frac{C_{op}(\gamma) K_{m + 1}}{m!} \epsilon^{\gamma - \frac{1}{2}} \, ,$$
where 
\begin{equation*} 
C_{op}(\gamma) := \begin{cases}
    \sqrt{\frac{27}{8}} (\gamma - \frac{1}{2})^{-\frac{3}{2}}, \quad \text{if} \quad \frac{1}{2} < \gamma \leq 2\\
    (\gamma -  1)^{-1}, \quad \text{if} \quad \gamma > 2
\end{cases} \, .
\end{equation*}
Further the $(m+1)$th order weak derivative of $f_t^\epsilon$ can be represented by
\begin{equation*}
\frac{\text{d}^{m + 1}}{\text{d} s^{m + 1}} f_t^\epsilon(s) := \begin{cases}
    K_{m + 1} (t - s)^{\gamma - 2 - m}, \quad \text{if} \quad s < t - \epsilon\\
    0, \quad \text{else}
\end{cases} \, ,
\end{equation*}
where $K_{m + 1} := (-1)^{m + 1} \Pi_{i = 1}^{m + 1}(\gamma - i)$. Since $\gamma < m + \frac{3}{2}$ by assumption, it then follows that 
\begin{align*} \norm{f_t^\epsilon}_{\mathcal{H}^{m + 1}([0,t))}^2 & = \int_0^{t - \epsilon} K_{m + 1}^2 (t - s)^{2 \gamma - 2m - 4} \, \text{d} s \\ 
& = \frac{K_{m + 1}^2}{2m + 3 - 2 \gamma} (\epsilon^{2 \gamma - 2 m - 3} - t^{2 \gamma - 2 m - 3}) \\ 
& \leq \frac{K_{m + 1}^2}{2m + 3 - 2 \gamma} \epsilon^{2 \gamma - 2 m - 3} < \infty \, ,
\end{align*}
where the last bound is possible due to the fact that $2m + 3 - 2 \gamma > 0$ since by assumption $\gamma < m + \frac{3}{2}$. Thus $f_t^\epsilon \in \mathcal{H}^{m + 1}([0,t))$. We can therefore apply the Bramble-Hilbert lemma as we did above to approximate $f_t^\epsilon$ by a piecewise polynomial function $g^\epsilon \in \mathcal{P}_m(\mathcal{I})$. Choosing $g = g^\epsilon$ we can therefore approximate $(II)$ by
\begin{align*}
    (II) = \norm{f_t^\epsilon - g^\epsilon}_{\mathcal{L}^2([0,t))} & \leq C h^{m + 1} \norm{f_t^\epsilon}_{\mathcal{H}^{m + 1}([0,t))}\\
    & \leq \frac{C \abs{K_{m + 1}}}{(2m + 3 - 2 \gamma)^{\frac{1}{2}}} \epsilon^{\gamma - m - \frac{3}{2}} h^{m + 1}\, .
\end{align*}
Putting $(I)$ and $(II)$ together and selecting $\epsilon = h$ we then get
\begin{align*} \norm{f_t - g^\epsilon}_{\mathcal{L}^2([0,t))} & \leq \norm{f_t - f_t^\epsilon}_{\mathcal{L}^2([0,t))} + \norm{f_t^\epsilon - g^\epsilon}_{\mathcal{L}^2([0,t))} \\
& \leq \abs{K_{m + 1}} \left(\frac{C_{op}(\gamma)}{m!} + \frac{C}{(2m + 3 - 2 \gamma)^{\frac{1}{2}}} \right) h^{\gamma - \frac{1}{2}} \, .
\end{align*}
Further we can approximate
\begin{align*}
\norm{f_t - \Pi_m f_t}_{\mathcal{L}^2([0,t))} & \leq (1 + C_\Pi) \norm{f_t - g^\epsilon}_{\mathcal{L}^2([0,t))} \\
& \leq (1 + C_\Pi) K_{m + 1} \left( \frac{C_{op}(\gamma)}{m!} + \frac{C}{(2m + 3 - 2 \gamma)^{\frac{1}{2}}} \right) h^{\gamma - \frac{1}{2}} \, .
\end{align*}
Thus again by Lemma \ref{lemma:function_approximation} 
\begin{align*} 
\left(\EV{\abs{c_k(t) - \Tilde{c}^{(m)}_k(t)}^2} \right)^{\frac{1}{2}} & \leq \frac{\sqrt{\lambda_k}}{\Gamma(\gamma)} \norm{f_t - \Pi_m f_t}_{\mathcal{L}^2([0,t))}  \\
& \leq \frac{\sqrt{\lambda_k}}{\Gamma(\gamma)} (1 + C_\Pi) \abs{K_{m + 1}}  \left( \frac{C_{op}(\gamma)}{m!} + \frac{C}{(2m + 3 - 2 \gamma)^{\frac{1}{2}}} \right) h^{\gamma - \frac{1}{2}}
\end{align*}

\textbf{Part 4:} Assume $\gamma = m + \frac{3}{2}$. We then repeat the proof for Case 3, but for $(II)$ we calculate
\begin{align*} \norm{f_t^\epsilon}_{\mathcal{H}^{m + 1}([0,t))}^2 & = \int_0^{t - \epsilon} K_{m + 1}^2 (t - s)^{-1} \, \text{d} s \\ 
& = K_{m + 1}^2 \log(t \epsilon^{-1}) < \infty \, ,
\end{align*}
so that $(II) = \norm{f_t^\epsilon - g^\epsilon}_{\mathcal{L}^2([0,t))} \leq C \abs{K_{m + 1}} \sqrt{\log(t \epsilon^{-1})} h^{m + 1}$. Thus, putting $(I)$ and $(II)$ together and selecting $h = \epsilon$, we get
\begin{align*} \norm{f_t - g^\epsilon}_{\mathcal{L}^2([0,t))} & \leq \norm{f_t - f_t^\epsilon}_{\mathcal{L}^2([0,t))} + \norm{f_t^\epsilon - g^\epsilon}_{\mathcal{L}^2([0,t))} \\
& \leq \abs{K_{m + 1}} \left(\frac{C_{op}(\gamma)}{m!} + C \sqrt{\log(t h^{-1})}\right) h^{m + 1} \\
& \leq \abs{K_{m + 1}} \max \left(\frac{C_{op}(\gamma)}{m!}, C\right) \left(1 +\sqrt{\log(t h^{-1})} \right) h^{m + 1} \, ,
\end{align*}
so that
\begin{align*} 
\left(\EV{\abs{c_k(t) - \Tilde{c}^{(m)}_k(t)}^2} \right)^{\frac{1}{2}} & \leq \frac{\sqrt{\lambda_k}}{\Gamma(\gamma)} \norm{f_t - \Pi_m f_t}_{\mathcal{L}^2([0,t))}  \\
& \leq \frac{\sqrt{\lambda_k}}{\Gamma(\gamma)} (1 + C_\Pi) \abs{K_{m + 1}} \max \left(\frac{C_{op}(\gamma)}{m!}, C\right) \left(1 +\sqrt{\log(t h^{-1})} \right) h^{m + 1} \, ,\end{align*}
\end{proof}

\begin{proof}[Proof of Lemma \ref{lem:projection-computation}] First observe that 
\begin{align*} \int_a^b s^k (t - s)^{\gamma - 1} \dint s &= \int_a^b (t - (t - s))^k (t - s)^{\gamma - 1} \dint s \\
&= \sum_{j = 0}^k \binom{k}{j} (-1)^j t^{k - j} \int_a^b (t - s)^j (t - s)^{\gamma - 1} \dint s \\
&= \sum_{j = 0}^k \binom{k}{j} \frac{(-1)^j t^{k - j}}{\gamma + j} \left[ (t - a)^{\gamma + j} - (t - b)^{\gamma + j} \right] = i_k \\
\end{align*}
Since $\{q_j \}_{j = 0}^m$ is a basis for $\mathbb{P}_m([a,b))$ we can express $p$ as $\sum_{j = 0}^m \beta_j q_j$. We can then calculate 
\begin{align*} \beta_k = \int_a^b p q_k \dint s = \int_a^b (t - s)^{\gamma - 1} q_k \dint s = \sum_{j = 0}^m q_k^{(j)} \int_a^b s^j (t - s)^{\gamma - 1} \dint s = \sum_{j = 0}^m q_k^{(j)} i_j \, .
\end{align*}
Then 
$$p = \sum_{k = 0}^m \beta_k q_k = \sum_{k = 0}^m \left(\sum_{l = 0}^m q_k^{(l)} i_l\right) \left(\sum_{j = 0}^m q_k^{(j)} s^j\right) = \sum_{j = 0}^m \left( \sum_{k = 0}^m \sum_{l = 0}^m q_k^{(l)} q_k^{(j)} i_l \right) s^j =: \sum_{j = 0}^m \alpha_j s^j \, ,$$
so that $\alpha_j = \sum_{k = 0}^m \sum_{l = 0}^m q_k^{(l)} q_k^{(j)} i_l$.
\end{proof}

We now turn to proving Lemma \ref{lem:boundedness-p0}. To prove this result, we first need another lemma.

\begin{lem} \label{lem:boundedness-p0-inequalities} Let $\Tilde{s}_\ell = t_{\ell - 1} + \theta (t_\ell - t_{\ell - 1})$, $\alpha > 0$, and ${\eta}: = \frac{\max_\ell(h_\ell)}{\min_\ell(h_\ell)}$. Fix $t_n \in \mathcal{I}$. If $\theta \in [0,1]$ then for all $0 < \ell < n$ we have that
\begin{align*}(t_n - \Tilde{s}_\ell)^\alpha & \leq (t_n - t_{\ell - 1})^{\alpha} \leq {\eta}^\alpha \left( \frac{2}{2 - \theta} \right)^\alpha (t_n - \Tilde{s}_\ell)^\alpha \\
(t_n - t_\ell)^\alpha & \leq (t_n - \Tilde{s}_\ell)^{\alpha} \leq {\eta}^\alpha \left( 2 - \theta \right)^\alpha (t_n - t_\ell)^\alpha
\end{align*}
Consequently, under stated assumptions, for all $\alpha \in \mathbb{R}$ 
\begin{align*}(t_n - \Tilde{s}_\ell)^\alpha & \lesssim (t_n - t_{\ell - 1})^{\alpha} \lesssim (t_n - \Tilde{s}_\ell)^\alpha \\
(t_n - t_\ell)^\alpha & \lesssim (t_n - \Tilde{s}_\ell)^{\alpha} \lesssim (t_n - t_\ell)^\alpha
\end{align*}

\begin{proof} We first assume that $\theta \in [0,1]$ and that $0 < \ell < n$. It is trivial that $(t_n - \Tilde{s}_\ell)^\alpha \leq (t_n - t_{\ell - 1})^{\alpha}$ and $(t_n - t_\ell)^\alpha \leq (t_n - \Tilde{s}_\ell)^{\alpha}$. Since $\ell < n$, then $t_n - t_\ell > 0$, so that
\begin{align*} \left( \frac{t_n - \Tilde{s}_\ell}{t_n - t_\ell} \right)^\alpha = \left(\frac{(1-\theta) h_\ell + \sum_{i = \ell + 1}^{n} h_i}{\sum_{i = \ell + 1}^{n} h_i} \right)^\alpha \leq {\eta}^\alpha \left(\frac{1 - \theta + n - \ell}{n - \ell} \right)^\alpha \leq {\eta}^\alpha \left( 2 - \theta \right)^\alpha \, .
\end{align*}
It follows that $(t_n - \Tilde{s}_\ell)^{\alpha} \leq {\eta}^\alpha \left( 2 - \theta \right)^\alpha (t_n - t_\ell)^\alpha$. Similarly, since $\ell < n$, then $t_n - \Tilde{s}_\ell > 0$, so that
\begin{align*} \left( \frac{t_n - t_{\ell - 1}}{t_n - \Tilde{s}_\ell} \right)^\alpha = \left( \frac{\sum_{i = \ell}^{n} h_i}{(1-\theta) h_\ell + \sum_{i = \ell + 1}^{n} h_i} \right)^\alpha \leq {\eta}^\alpha \left(\frac{n - \ell + 1}{ 1 - \theta + n - \ell} \right)^\alpha \leq {\eta}^\alpha \left( \frac{2}{2 - \theta} \right)^\alpha \, .
\end{align*}
It follows that $(t_n - t_{\ell - 1})^{\alpha} \leq {\eta}^\alpha \left( \frac{2}{2 - \theta} \right)^\alpha (t_n - \Tilde{s}_\ell)^\alpha$. The conclusion follows.
\end{proof}

\end{lem}

\begin{proof}[Proof of Lemma \ref{lem:boundedness-p0}] Let $h(s) = g(s) + \alpha f_{t_n}(s)$, where $g \in \mathcal{P}_0(\mathcal{I})$, and $f_t(s) = (t - s)^{\gamma - 1}$. Then $g$ is of the form $g(s) = \sum_{\ell = 1}^{N} g_\ell \mathbb{I}_{[t_{\ell - 1}, t_\ell)}(s)$, so that
$$\norm{g + \alpha f_{t_n}}_{\mathcal{L}^2([0,t_n))}^2 = \sum_{\ell = 1}^n \left[ g_\ell^2 (t_\ell - t_{\ell - 1}) + 2 \alpha g_\ell \int_{t_{\ell - 1}}^{t_\ell} f_{t_n}(s) \, \text{d}s + \alpha^2 \int_{t_{\ell - 1}}^{t_\ell} f_{t_n}^2 \, \text{d}s  \right] \, .$$
Also 
$$\norm{\Pi_0(g + \alpha f_t)}_{\mathcal{L}^2([0,t_n))}^2 = \sum_{\ell = 1}^n (t_\ell - t_{\ell - 1}) \left[ g_\ell^2 + 2 \alpha g_\ell (t - \Tilde{s}_\ell)^{\gamma - 1} + \alpha^2 (t - \Tilde{s}_\ell)^{2\gamma - 2} \right] \, ,$$
where $\Tilde{s}_\ell := t_{\ell - 1} + \theta (t_\ell - t_{\ell - 1})$. We now consider
\begin{align*} 2 \alpha g_\ell (t_n - \Tilde{s}_\ell)^{\gamma - 1} (t_\ell - t_{\ell - 1}) &= 2 \alpha g_\ell ((t_n - \Tilde{s}_\ell)^{\gamma - 1}(t_n - t_{\ell - 1}) - (t_n - \Tilde{s}_\ell)^{\gamma - 1}(t_n - t_{\ell})) = (*)
\end{align*}
If $\ell < n$, then by Lemma \ref{lem:boundedness-p0-inequalities} we have that 
$$(*) \lesssim 2 \alpha g_\ell ((t_n - t_{\ell - 1})^{\gamma} - (t_n - t_\ell)^{\gamma}) = 2 \alpha g_\ell \gamma \int_{t_{\ell - 1}}^{t_\ell} f_{t_n}(s) \, \text{d}s\, .$$
If $n = \ell$, then $t_n - t_{\ell} = 0$ and thus 
$$(*) = 2 \alpha g_\ell (1 - \theta)^{\gamma - 1} (t_n - t_{n - 1})^{\gamma} =  2 \alpha g_\ell (1 - \theta)^{\gamma - 1} \gamma \int_{t_{n - 1}}^{t_n} f_{t_n}(s) \, \text{d}s\, .$$
Note that if $\gamma < 1$, then we must assume $\theta \neq 1$, else $(1 - \theta)^{\gamma - 1}$ will be undefined. We can use an almost identical argument to show that $\alpha^2 (t - \Tilde{s}_\ell)^{2\gamma - 2} \lesssim \alpha^2 \int_{t_{\ell - 1}}^{t_\ell} f_t(s)^2 \, \text{d}s$. Thus, if $\theta \in [0,1]$ and $\gamma \leq 1$, or if $\theta \in [0,1)$, we have that
\begin{align*}
\norm{\Pi_0(g + \alpha f_{t_n})}_{\mathcal{L}^2([0,t_n))}^2 & = \sum_{\ell = 1}^n (t_\ell - t_{\ell - 1}) \left[ g_\ell^2 + 2 \alpha g_\ell (t_n - \Tilde{s}_\ell)^{\gamma - 1} + \alpha^2 (t_n - \Tilde{s}_\ell)^{2\gamma - 2} \right] \\
& \leq C_{\Pi}^2 \sum_{\ell = 1}^n \left[ g_\ell^2 (t_\ell - t_{\ell - 1}) + 2 \alpha g_\ell \int_{t_{\ell - 1}}^{t_\ell} f_{t_n}(s) \, \text{d}s + \alpha^2 \int_{t_{\ell - 1}}^{t_\ell} f_{t_n}(s)^2 \, \text{d}s  \right] \\
& = C_{\Pi}^2 \norm{g + \alpha f_{t_n}}_{\mathcal{L}^2([0,t_n))}^2 \, ,
\end{align*}
where $C_\Pi$ depends on $\gamma$, $\theta$ and ${\eta} = \frac{\max_\ell(h_\ell)}{\min_\ell(h_\ell)}$. 
\end{proof}

\begin{proof}[Proof of Corollary \ref{cor:opt-rate-balancing}]
From {Theorem \ref{thm:main-result}} we know that the error $\epsilon \lesssim M^{-\frac{2 \nu}{d}} + C(M, \gamma) h^b$, where $b := \min(2 \gamma - 1, 2 m + 2)$ and
$$C(M, \gamma) = \begin{cases}
    \sum_{k = 1}^M \lambda_k & \text{when} \, \gamma \in (\frac{1}{2}, m + \frac{3}{2}] \, , \\
    \sum_{k = 1}^M \lambda_k \mu_k^{-2 \gamma + 2m + 3} & \text{when} \, \gamma \in (m + \frac{3}{2}, m + 2] \, , \\
    \sum_{k = 1}^M \lambda_k \mu_k^{-1} & \text{when} \, \gamma \in (m + 2, \infty) \, .
\end{cases} \, .$$
Note that $C(M, \gamma) \lesssim \sum_{k = 1}^M k^{\delta - 1}$ by {Assumption \ref{assumption:operator}}. If {$\delta < 0$} this sum converges, so $C(M, \gamma) \lesssim 1$. If $0 < \delta < 1$, then 
$$C(M, \gamma) \lesssim 1 + \int_1^M k^{\delta - 1} = 1 - \frac{1}{\delta} + \frac{1}{\delta} M^\delta \lesssim M^\delta \, .$$ 
If $\delta = 0$, then similarly $C(M, \gamma) \lesssim 1 + \ln(M) \lesssim \ln(M)$. Finally if $\delta \geq 1$, then 
$C(M, \gamma) \lesssim \sum_{k = 1}^M k^{\delta - 1} \leq \sum_{k = 1}^M M^{\delta - 1} = M^\delta$. Thus
$$\epsilon \lesssim \begin{cases}
    M^{-\frac{2 \nu}{d}} + M^\delta h^b & \text{when} \, \delta > 0 \, , \\
    M^{-\frac{2 \nu}{d}} + \ln(M) h^b & \text{when} \, \delta = 0 \, , \\
    M^{-\frac{2 \nu}{d}} + h^b & \text{when} \, \delta < 0 \, , \\
\end{cases} \, .$$
Note that we can always achieve the optimal rate $M^{-\frac{2 \nu}{d}}$ by setting $h = M^{-\zeta}$ and choosing $\zeta$ appropriately. The optimal choice of $\zeta$ is given by
$$\zeta = \begin{cases}
    \frac{\delta}{b} + \frac{2 \nu}{d b}& \text{when} \, \delta > 0 \, , \\
    \frac{2 \nu}{d b} & \text{when} \, \delta \leq 0 \, , \\
\end{cases} \, .$$
though note that in the case $\delta = 0$ this choice of $\zeta$ only gives us the optimal rate up to $\ln(M)$. To compute $\zeta$ we will consider the cases $\gamma \in (\frac{1}{2}, m + \frac{3}{2})$, $\gamma \in (m + \frac{3}{2}, m + 2]$, and {$\gamma \in (m + 2, \infty)$} separately. 
We first consider the case $\gamma \in (\frac{1}{2}, m + \frac{3}{2})$, where $\delta = 1 + r_\lambda $. The above then becomes
\begin{align*}
    \zeta &= \begin{cases}
    \frac{1 + r_\lambda}{2 \gamma - 1} - \frac{1 + r_\lambda - (2\gamma - 1) r_\mu}{2 \gamma - 1} & \text{when} \, \delta > 0 \, , \\
    - \frac{1 + r_\lambda - (2\gamma - 1) r_\mu}{2 \gamma - 1} & \text{when} \, \delta \leq 0 \, , \\
\end{cases} \\ 
&= \begin{cases}
    r_\mu & \text{when} \, \delta > 0 \, , \\
    -\frac{\delta}{b} + r_\mu & \text{when} \, \delta \leq 0 \, , \\
\end{cases} \, .
\end{align*}
Secondly we have the case $\gamma \in (m + \frac{3}{2}, m + 2)$, where $\delta = 1 + r_\lambda - (2 \gamma - 2 m - 3) r_\mu$. The above then becomes
\begin{align*}
    \zeta &= \begin{cases}
    \frac{1 + r_\lambda - (2 \gamma - 2 m - 3) r_\mu}{2 m + 2} - \frac{1 + r_\lambda - (2\gamma - 1) r_\mu}{2 m + 2} & \text{when} \, \delta > 0 \, , \\
    - \frac{1 + r_\lambda - (2\gamma - 1) r_\mu}{2 m + 2} & \text{when} \, \delta \leq 0 \, , \\
\end{cases} \\ 
&= \begin{cases}
    r_\mu & \text{when} \, \delta > 0 \, , \\
    -\frac{\delta}{b} + r_\mu & \text{when} \, \delta \leq 0 \, , \\
\end{cases} \, .
\end{align*}
Finally we have the case $\gamma \in (m + 2, \infty)$, where $\delta = 1 + r_\lambda - r_\mu$. The above then becomes
\begin{align*}
    \zeta &= \begin{cases}
    \frac{1 + r_\lambda - r_\mu}{2 m + 2} - \frac{1 + r_\lambda - (2\gamma - 1) r_\mu}{2 m + 2} & \text{when} \, \delta > 0 \, , \\
    - \frac{1 + r_\lambda - (2\gamma - 1) r_\mu}{2 m + 2} & \text{when} \, \delta \leq 0 \, , \\
\end{cases} \\ 
&= \begin{cases}
    \frac{2 \gamma - 2}{2 m + 2} r_\mu & \text{when} \, \delta > 0 \, , \\
    -\frac{\delta}{b} + \frac{2 \gamma - 2}{2 m + 2} r_\mu & \text{when} \, \delta \leq 0 \, , \\
\end{cases} \, .
\end{align*}
We summarize this as
{\begin{equation*}
    \zeta = \begin{cases}
    r_\mu - \min(0, \frac{\delta}{b}) & \text{when} \, \gamma \leq m + 2 \, , \\
    \frac{2 \gamma - 2}{2 m + 2} r_\mu - \min(0,\frac{\delta}{b}) & \text{when} \,  \gamma > m + 2 \, ,
\end{cases} \, . \qedhere
\end{equation*}}
\end{proof}

\begin{proof}[Proof of Corollary \ref{cor:comp-cost}] 
We know from Corollary \ref{cor:opt-rate-balancing} that if $h = M^{-\zeta}$ then the error $\epsilon \lesssim M^{-\frac{2 \nu}{d}}$, or equivalently $\epsilon^{-\frac{d}{2 \nu}} \gtrsim M$. {Since the complexity is $O(JMN) = O(J M h^{-1})$, we thus have
\begin{equation*}
    \mathcal{C} \lesssim J M h^{-1} = J M^{1 + \zeta} \lesssim J \epsilon^{-\frac{d}{2 \nu}(1 + \zeta)} \, . \qedhere
\end{equation*}}
\end{proof}

\section{Data availability statement} Supporting code is available at https://doi.org/10.5281/zenodo.18659877.

\section{Acknowledgments} I thank my supervisors Espen R. Jakobsen and Geir-Arne Fuglstad for invaluable guidance and insightful discussions.

\section{Declarations} 
\noindent
\textbf{Funding} \, Not applicable.

\noindent
\textbf{Conflicts of interest} \, The author declares no conflict of interest.

\bibliographystyle{siam} 
\bibliography{Manuscript}

\end{document}